\documentclass[10pt]{article}

\usepackage{amsmath}
\usepackage{amsfonts}
\usepackage{amsthm}
\usepackage{graphicx}
\usepackage{caption}
\usepackage{subcaption}
\usepackage{float}
\usepackage{enumerate}
\newtheorem{theorem}{Theorem}
\newtheorem{lemma}[theorem]{Lemma}

\title{Adaptive mixed GMsFEM for flows in heterogeneous media}

\author{Ho Yuen Chan\thanks{Department of Mathematics, The Chinese University of Hong Kong, Shatin, Hong Kong SAR} \and 
Eric Chung\thanks{Department of Mathematics, The Chinese University of Hong Kong, Shatin, Hong Kong SAR} \and Yalchin Efendiev\thanks{Department of Mathematics  and Institute for Scientific Computation (ISC),
Texas A\&M University,
College Station, Texas 77843-3368, USA}}

\begin{document}

\maketitle

\begin{abstract}
In this paper, we present two adaptive methods for the basis enrichment of the mixed Generalized Multiscale Finite Element Method (GMsFEM) for solving the flow problem in heterogeneous media. We develop an a-posteriori error indicator which depends on the norm of a local residual operator. Based on this indicator, we construct an offline adaptive method to increase the number of basis functions locally in coarse regions with large local residuals. We also develop an online adaptive method which iteratively enriches the function space by adding new functions computed based on the residual of the previous solution and special minimum energy snapshots. We show theoretically and numerically the convergence of the two methods. The online method is, in general, better than the offline method as the online method is able to capture distant effects (at a cost of online computations), and both methods have faster convergence than a uniform enrichment. Analysis shows that the online method should start with certain number of initial basis functions in order to have the best performance. The numerical results confirm this and show further that with correct selection of initial basis functions, the convergence of the online method can be independent of the contrast of the medium. We consider cases with both very high and very low conducting inclusions and channels in our numerical experiments.
\end{abstract}

\section{Introduction}

Many real-world problems involve multiple scales and high contrast. To solve these problems, we often adopt some forms of model reduction such as upscaling and multiscale methods. These methods can reduce the degrees of freedom of a problem. For example, in upscaling methods \cite{durlofsky1991numerical, wu2002analysis,numerical-homo,2d-waves}, the multiscale media are upscaled so that the problem can be solved on a coarse grid. In multiscale methods \cite{arbogast2004analysis, chu2010new, engquist2013heterogeneous, efendiev2011multiscale, efendiev2009multiscale, efendiev2004multiscale, ghommem2013mode, chung2010reduced, chung2011energy, chung2014generalized, chung2013sub,GMsFEM-elastic,elastic-jcp,aarnes04,jennylt03}, basis functions are solved on a fine grid to capture the multiscale features of a medium and the problem is then solved on the coarse grid with these basis functions.

In this paper, we will present two adaptive enrichment algorithms for the generalized multiscale finite element method (GMsFEM) in solving the mixed framework of the flow problem in heterogeneous media \cite{chung2014mixed}. The first method is based on a local error indicator. We use this indicator to search for the regions, where more basis functions are needed. This method will only add pre-computed basis functions, which are computed in the offline stage so we call it an offline adaptive method. In the second method, new basis functions are computed based on the previous solutions. We call it an online adaptive method.

GMsFEM is a generalization of the classical multiscale finite element method \cite{hou1997multiscale}. In the classical method, one basis function per coarse edge is used to capture the multiscale features. For the multiscale mixed finite element method, one may see \cite{chen2003mixed, arbogast2000numerical,aarnes04}. GMsFEM allows more basis functions per coarse edge to be used to take into account the effects of non-separable scales. The main idea is to solve local spectral problems for the selection of basis functions. The formation of basis functions in GMsFEM can be divided into offline and online stages. In the offline stage, offline basis functions are computed based on the multiscale features so that these functions can be reused for any input parameters to solve the equation. Online functions are those depending on the parameters. In \cite{chung2014adaptive}, an adaptive algorithm is developed to enrich the space by adding basis functions which are formed in the offline stage. In \cite{chung2015residual}, adaptive methods which involve the formation of new online basis functions based on the previous solution are developed. These methods show significant acceleration in the convergence rate of GMsFEM. There are also related methods developed for the discontinuous Galerkin formulation in \cite{chung2014GMsDGM} and \cite{chung2015online}.

In the paper, we will focus on the mixed framework of the flow problem. The mixed methods are important for many applications, such as flows in porous media, where the mass conservation is essential.  We developed two adaptive methods to enrich the function space. One involves only offline basis functions while the other adds new online basis functions that are constructing using special minimum energy snapshots. We call them an offline and an online adaptive methods respectively. Two local spectral problems are developed for constructing multiscale basis functions. Both of them can be used in the online method, but only one can be used in the offline method. We propose error indicators which are based on the $\mathcal{L}^2$ and the $H(\text{div})$ norms of the local residual. These error indicators can be used to approximate the error of the solution. From \cite{chung2014mixed}, we know that the error between the GMsFEM solution and the fine grid solution involves two parts: one due to the selection of the basis functions and the other due to the discretization of the source function.  In this paper, we will assume the error due to the discretization of the source function is small and consider only the former part. The offline adaptive method depends on the error indicator to help the selection of basis functions. The online adaptive method produces new basis functions iteratively by projecting the previous solution on the space of divergence free functions. We emhasize that \cite{chung2014mixed} gives a-priori error estimate of the mixed GMsFEM, and the purposes of this paper are a-posteriori error estimates and adaptivity.

In our analysis, we prove the convergence of the two methods. It can be shown that the error is bounded by the local error indicators. By adding offline basis functions to those coarse grid edges with large error indicator, we can show a guaranteed convergence rate for the error of the solution together with the local error indicators. The convergence rate depends on the parameters of the offline adaptive method. These parameters control the number of coarse grid edges to be chosen and the number of basis functions to be added for each of those edges. For the online adaptive method, a set of non-overlapping subsets of the domain is selected. New basis functions are computed on each of these subsets. We show that the convergence rate depends on the norm of the residual operator restricted on those subsets, and also the eigenvalues of the offline basis functions that are not included in the initial basis.

We present some numerical results to show the convergence behaviour and some properties of the adaptive methods. We consider both high and low conductivity inclusions and channels in the domain. By comparing the adaptive methods to uniformly enriching the function space, one can see the efficiency of the adaptive methods. In particular, the performance of the online adaptive method is generally the best since it adds functions which are computed based on the previous solution while the offline adaptive method enriches the space by adding basis functions which are independent of the input parameters. We will see that both the choice of the non-overlapping regions and the initial number of basis functions on each coarse grid edge affect the convergence of the online adaptive method. Some of the eigenvalues from the spectral problems are sensitive to the contrast of the problem. By including functions corresponding to those eigenvalues in the initial basis, the convergence of 
online adaptive method becomes independent of the contrast.

The rest of the paper is organized in the following way. In the next section, we briefly introduce the basic idea of mixed GMsFEM. At the end of the section, we give the detail of the two adaptive methods. In Section \ref{section_analysis}, we state and prove the convergence results for the adaptive methods. In Section \ref{section_numerical}, numerical results are given to illustrate the convegence behaviour of the adaptive method and the factors affecting the convergence. The paper ends with a conclusion.

\section{Method description}

\subsection{Overview}
Consider the high-contrast flow problem in a mixed formulation:
\begin{equation}\label{equation}
\begin{aligned}
    \kappa^{-1}v + \nabla p &= 0 &&\mbox{in } D \\
    \mbox{div}(v) &= f &&\mbox{in } D, \\
\end{aligned}
\end{equation}
with Neumann boundary condition $v \cdot n = g$ on $\partial D$, where $\kappa$ is a high-contrast permeability field, $D$ is the computational domain in $\mathbb{R}^{n}$ and $n$ is the unit outward normal vector of the boundary of $D$.

We will solve the equation on two meshes with different scales. Let $\mathcal{T}^{H}$ be a partition of $D$ into finite elements (triangles, quadrilaterals, tetrahedra, etc.), where $H$ is the mesh size. We call $\mathcal{T}^{H}$ the coarse grid. Next we construct a finer grid. For each coarse grid element $K \in \mathcal{T}^{H}$, we further partition $K$ into a finer mesh such that the resulting partition $\mathcal{T}^{h}$ of $D$ with size $h$ is conforming across coarse-grid edges. We call $\mathcal{T}^{h}$ the fine grid. Denote the set of all faces of the coarse grid as $\mathcal{E}^{H}$, and let $N_{e}$ be the total number of faces of the coarse grid. We define the coarse grid neighborhood $\omega_{i}$ of a face $E_i \in \mathcal{E}^{H}$ as
$$\omega_{i} = \bigcup \{K \in \mathcal{T}^{H} : E_{i} \in \partial K \},$$
which is indeed a union of two coarse grid blocks.

Next, we define the notations for the solution spaces for pressure and velocity. Let $Q$ be the space of functions which are constant on each coarse grid block. {\it We will use this space to approximate $p$}. For the velocity space, we will first construct a set of basis functions $\beta^{(i)}_{\text{snap}}$ for each coarse grid neighborhood $\omega_{i}$. We call $V_{\text{snap}} = \bigoplus_{E_{i} \in \mathcal{E}^{H}} V_{\text{snap}}^{(i)}$ the snapshot space, where $V_{\text{snap}}^{(i)} =  \mbox{span} \left( \beta^{(i)}_{\text{snap}} \right)$. The snapshot space is an extensive set of functions which can be used to approximate the solution $v$. However, this space is large and we will reduce it to a smaller one before we solve the equation. From each $V^{(i)}_{\text{snap}}$, we select a set of basis functions $\beta^{(i)}_{\text{ms}}$. Denote $V_{\text{ms}}^{(i)} = \mbox{span} \left(\beta_{\text{ms}}^{(i)} \right)$ and $V_{\text{ms}} = \bigoplus_{E_{i} \in \mathcal{E}^{H}} V^{(i)}_{\text{ms}}$. The size of $V_{\text{ms}}$ is generally smaller than $V_{\text{snap}}$. We will use the space $V_{\text{ms}}$ to approximate the velocity $v$.

With the pressure space $Q$ and the velocity space $V_{\text{ms}}$, we solve for $p_{\text{ms}} \in Q$ and $v_{\text{ms}} \in V_{\text{ms}}$ such that
\begin{equation}\label{MsEquation}
\begin{aligned}
\int_{D} \kappa^{-1} v_{\text{ms}} \cdot w - \int_{D} \mbox{div}(w) p_{\text{ms}} &= 0 && \forall w \in V^{0}_{\text{ms}}, \\
\int_{D} \mbox{div}(v_{\text{ms}}) q &= \int_{D} fq && \forall q \in Q,
\end{aligned}
\end{equation}
with boundary condition $v_{\text{ms}} \cdot n = g_{H}$ on $\partial D$, where $V^{0}_{\text{ms}} = \{ v \in V_{\text{ms}} : v \cdot n = 0 \mbox{ on } \partial D\}$, and $g_{H}$ is the projection of $g$ in the sense that
\begin{equation*}
\int_{E_{i}} (g_{H} - g) \phi \cdot n = 0 \qquad \forall \phi \in \beta^{(i)}_{\text{snap}} \mbox{ and } E_{i} \subseteq \partial D,
\end{equation*}
and $g_H$ is constant on each fine grid face. 

For $\Omega \subset D$ and $v \in V_{\text{snap}}$, we define two norms $\|v\|_{\mathcal{L}^2(\Omega; \kappa^{-1})}$ and $\|v\|_{H(\text{div}; \Omega; \kappa^{-1})}$ as
$$ \|v\|_{\mathcal{L}^2(\Omega; \kappa^{-1})} = \left( \int_{\Omega} \kappa^{-1} |v|^2 \right)^{\frac{1}{2}}$$
and
$$ \|v\|_{H(\text{div}; \Omega; \kappa^{-1})} = \left( \int_{\Omega} \kappa^{-1} |v|^2 + \int_{\Omega} \text{div}(v)^2 \right)^{\frac{1}{2}}.$$
We will use these two norms in the adaptive methods.

In the coming sections, we will discuss the formation of the snapshot space $V_{\text{snap}}$ and the method of selecting the $\beta_{\text{ms}}^{(i)}$'s. We will also give two adaptive methods of enrichment of the multiscale space $V_{\text{ms}}$ so as to get a more accurate solution without using too many basis functions.

\subsection{Snapshot space}
In this section, we will present the construction of the snapshot space which is a large function space containing basis functions up to the resolution of fine grid faces on the coarse grid faces. We construct the local snapshot bases $\beta_{\text{snap}}^{(i)}$ by solving a local problem on each coarse grid neighbourhood $\omega_{i}$, and then combine the $\beta_{\text{snap}}^{(i)}$'s to form the snapshot space $V_{\text{snap}}$.

Let $E_{i} \in \mathcal{E}^{H}$. We write $E_{i} = \bigcup^{J_{i}}_{j = 1}e_{j}$, where the $e_{j}$'s are the fine grid faces contained in $E_{i}$ and $J_{i}$ is the total number of those fine grid faces. We will solve the following local problem to obtain $\beta^{(i)}_{\text{snap}}$,
\begin{equation} \label{localEquation}
\begin{aligned}
\kappa^{-1} v^{(i)}_{j} + \nabla p^{(i)}_{j} &= 0 &&\mbox{in } \omega_{i} , \\
\mbox{div} (v^{(i)}_{j}) &= \alpha^{(i)}_{j} &&\mbox{in } \omega_{i} ,
\end{aligned}
\end{equation}
subject to the homogeneous Neumann boundary condition $v^{(i)}_{j} \cdot n_{i} = 0$ on $\partial \omega_{i}$. We want the snapshot basis to contain solutions of the local problem with all possible boundary conditions on the edge $E_{i}$ up to the fine grid resolution. Therefore, the problem is solved separately on each coarse-grid element $K \subseteq \omega_{i}$ with the additional boundary condition $v_{(i)} \cdot n_{i} = \delta^{(i)}_{j}$ on $E_{i}$, where $\delta^{(i)}_{j}$ is defined by
\begin{equation*}
\delta^{(i)}_{j} =
\begin{cases}
1 &\mbox{on } e_{j}, \\
0 &\mbox{on } E_{i}\backslash e_{j},
\end{cases}
\end{equation*}
and $n_i$ is a fixed unit-normal vector for each face $E_i$. The function $\alpha_{j}^{(i)}$ is constant on each coarse grid block and it should satisfy the condition $\int_{K} \alpha_{j}^{(i)} = \int_{\partial K} v_{j}^{(i)} \cdot n_{i}$ for every $K \subseteq \omega_{i}$.

The set of solutions to (\ref{localEquation}) is the local snapshot basis $\beta_{\text{snap}}^{(i)}$. Using $\bigcup _{E_i \in \mathcal{E}^H} \beta_{\text{snap}}^{(i)}$ as a basis, we have the snapshot space $V_{\text{snap}}$.

\subsection{Offline space}
After we obtain the snapshot spaces $V_{\text{snap}}$, we perform a dimension reduction to get a smaller space. Such reduced space is called the offline space. The reduction is achieved by solving a local spectral problem on each coarse grid neighborhood $\omega_i$. Some of the eigenfunctions will be used to form the local basis $\beta_{\text{ms}}^{(i)}$. The local spectral problem is to find real number $\lambda$ and $v \in V$ such that
\begin{equation}\label{spectralProblem}
a(v, w) = \lambda s(v, w), \qquad \forall w \in V,
\end{equation}
where $V$ is the snapshot space, $a$ and $s$ are symmetric positive definite bilinear operators defined on $V \times V$.

We propose the following two possible spectral problems for the basis selection. In these two problems, $V$ is set to be $V_{\text{snap}}^{(i)}$.

\paragraph{Spectral problem 1:} We take
\begin{equation*}
\begin{aligned}
a_i(v, w) &= \int_{E_{i}} \kappa^{-1} (v \cdot n_{i})(w \cdot n_{i}), \\
s_i(v, w) &= \frac{1}{H} \left( \int_{\omega_{i}} \kappa^{-1} v \cdot w + \int_{\omega_{i}} \mbox{div}(v) \mbox{div}(w) \right),
\end{aligned}
\end{equation*}
for $v,\,w \in V_{\text{snap}}^{(i)}$, where $n_{i}$ is a fixed unit normal vector on $E_{i}$. The term $1/H$ is added so that $a_i$ and $s_i$ have the same scale. 

\paragraph{Spectral problem 2:} For $v \in V_{\text{snap}}^{(i)}$, we define $\tilde{v}$ to be the extension of $v \cdot n_i$ in $\bigoplus_{\omega_j \cap \omega_i \neq \phi} V_{\text{snap}}^{(j)}$ by minimizing the energy norm on $\omega_i$, i.e., we find $\tilde{v} \in \bigoplus_{\omega_j \cap \omega_i \neq \phi} V_{\text{snap}}^{(j)}$ such that $\tilde{v} \cdot n_i = v \cdot n_i$ on $E_i$ and $\|\tilde{v}\|_{\mathcal{L}^2(\omega_i; \kappa^{-1})} \leq \|\varphi\|_{\mathcal{L}^2(\omega_i; \kappa^{-1})}$ for all $\varphi \in \bigoplus_{\omega_j \cap \omega_i \neq \phi} V_{\text{snap}}^{(j)}$.
We take
\begin{equation*}
\begin{aligned}
a_i(v, w) &= \int_{\omega_{i}} \kappa^{-1} \tilde{v} \cdot \tilde{w}, \\
s_i(v, w) &= \int_{\omega_{i}} \kappa^{-1} v \cdot w.
\end{aligned}
\end{equation*}
Remark that the eigenvalues of this spectral problem is always bounded above by 1.

\paragraph{}
Note that spectral problem 1 involves only the space $V_{\text{snap}}^{(i)}$. However, in spectral problem 2, the calculation of $\tilde{v}$ involves all the $V_{\text{snap}}^{(j)}$'s such that $\omega_j$ and $\omega_i$ have non-empty intersection, and so spectral problem 2 requires slightly more calculation.
We note that $\tilde{v}$ is the minimum energy extension of $v\cdot n_i|_{E_i}$ in the space $\bigoplus_{\omega_j \cap \omega_i \neq \phi} V_{\text{snap}}^{(j)}$.

After solving the spectral problem in a coarse grid neighborhood $\omega_i$, we arrange the eigenfunctions $\phi_j^{(i)}$ in ascending order of the corresponding eigenvalues
$$ \lambda_1^{(i)} \leq \lambda_2^{(i)} \leq \cdots \leq \lambda_{J_i}^{(i)}.$$
We then let $\beta_{\text{ms}}^{(i)}$ be the set of the first $l_i$ eigenfunctions, where $l_i$ is a predefined integer or $l_i$ is depending on the eigenvalues. The selection of $l_i$ will be discussed in the analysis section. When the spectral problem is specified and the $l_i$'s are selected, we construct the offline space $V_{\text{ms}}$.

Suppose $l_i$ is fixed for each coarse grid neighborhood $\omega_i$. We define $\Lambda_{\min} = \min_{E_i \in \mathcal{E}^H} \lambda_{l_i+1}^{(i)}$. The value 
of   $(\Lambda_{\min})^{-1}$ indicates the error between the multiscale solution and the solution obtained by using the whole snapshot space. Therefore, we want $(\Lambda_{\min})^{-1}$ to be as small as possible. However, in spectral problem 2, we can see that $(\Lambda_{\min})^{-1}$ is bounded below by 1. Therefore, this spectral problem will only be used in the online adaptive method,
and is shown to be crucial in selecting initial bases. 

\subsection{Offline adaptive method}
In this section, we will introduce an error indicator on each coarse grid neighborhood. Based on this estimator, we develop an offline adaptive enrichment method to solve equation (\ref{equation}) iteratively by adding offline basis functions supported on some coarse grid neighborhoods in each iteration. In this offline adaptive method, we will use spectral problem 1.

For each coarse grid neighborhood $\omega_i$, we define the residual operator $R_i$ as a linear functional on $V_{\text{snap}}^{(i)}$ by
$$ R_i(v) = \int_{\omega_i} \kappa^{-1} v_{\text{ms}} \cdot v - \int_{\omega_i} \mbox{div} (v) p_{\text{ms}} \qquad \forall v \in V_{\text{snap}}^{(i)},$$
where $(v_{\text{ms}},\, p_{\text{ms}})$ is the solution pair of (\ref{MsEquation}).

We take $\| R_i \|_{(V_{\text{snap}}^{(i)})^*}^2 (\lambda_{l_i+1}^{(i)})^{-1}$ as our error indicator, where
$$\|R_i\|_{{V_{\text{snap}}^{(i)}}^*} = \sup_{v \in V_{\text{snap}}^{(i)}} \frac{|R_i(v)|}{\|v\|_{H(\text{div}; \omega_i; \kappa^{-1})}}.$$
This quantity can be used to approximate the error since we have
$$ \|v_{\text{snap}} - v_{\text{ms}}\|_{\mathcal{L}^2(D; \kappa^{-1})}^2 \leq C_{\text{err}} \sum_{E_i \in \mathcal{E}^H} \|R_i\|_{{V_{\text{snap}}^{(i)}}^*}^2 (\lambda_{l_i+1}^{(i)})^{-1},$$
where $v_{\text{snap}} \in V_{\text{snap}}$ is the solution obtained by solving the equation using the whole snapshot space $V_{\text{snap}}$. We will prove this inequality in the next section. 

With this error indicator, we now present the offline adaptive method. In this method, only spectral problem 1 will be used. We let $m \geq 0$ denote the enrichment level. For each coarse grid neighborhood $\omega_i$, we use $l_i^m$ to represent the number of eigenfunctions used to form $V_{\text{ms}}^{(i)}$. We use $V_{\text{ms}}^m$ to denote the space $V_{\text{ms}}$ at enrichment level $m$.

\paragraph{Offline adaptive method:} Fix the number $\theta$ and $\delta_0$ with $0 < \theta,\, \delta_0 < 1$. We start with iteration number $m = 0$. Fix an initial number of offline basis functions $l_i^0$ for each coarse grid neighborhood to form the offline space $V_{\text{ms}}^0$. Then, we go to step 1 below.

\begin{description}
\item[Step 1:] Find the multiscale solution. We solve  for the multiscale solution $v_{\text{ms}}^{m} \in V_{\text{ms}}^{m}$ and $p_{\text{ms}}^{m} \in Q$ satisfying
\begin{equation} \label{MsEquation_offline}
\begin{aligned}
\int_{D} \kappa^{-1} v_{\text{ms}}^m \cdot w - \int_{D} \mbox{div}(w) p_{\text{ms}}^m &= 0 && \forall w \in (V_{\text{ms}}^m)^0, \\
\int_{D} \mbox{div}(v_{\text{ms}}^m) q &= \int_{D} fq && \forall q \in Q.
\end{aligned}
\end{equation}
\item[Step 2:] Compute the error indicators. For each coarse grid neighborhood $\omega_{i}$, we compute the local residual $\eta_{i}^2 = \|R_{i}\|_{{V_{\text{snap}}^{(i)}}^{*}}^2 (\lambda_{l_i+1}^{(i)})^{-1}$ and rearrange the residuals in decreasing order $\eta_{1} \geq \eta_{2} \geq \cdots \geq \eta_{N_e}$.
\item[Step 3:] Select the coarse grid neighborhoods where basis enrichment is needed. We take the smallest $k$ such that
$$ \theta^2 \sum_{i=1}^{N_e} \eta_{i}^2 \leq \sum_{i=1}^{k} \eta_{i}^2.$$
We will enrich the offline space by adding basis functions which are supported in the coarse grid neighborhoods $\omega_1, \dots ,\, \omega_k$.
\item[Step 4:] Add basis functions to the space. For each of the selected coarse grid neighborhood $\omega_i$ from step 3, we will take the smallest $s_i$ such that $l_{l_i^m+1}^{(i)} / \lambda_{l_i^m+s_i+1}^{(i)} \leq \delta_0$. We then set $l_i^{m+1} = l_i^m + s_i$ so that $s_i$ more basis functions are included to form a larger space $V_{\text{ms}}^{m+1}$. For the other neighborhoods, we set $l_i^{m+1} = l_i^m$. We will explain the reason for choosing such $s_i$  in the analysis section. 
\end{description}

After Step 4, we repeat from Step 1 until the global error indicator $\sum_{i=1}^{N_e} \eta_{i}^2$ is small enough or the total number of bases reaches certain level. The calculations of all the local error indicators can be time consuming. However, since the error indicators are independent of each other, the computation can be done in a parallel approach.

\subsection{Online adaptive method}
Next, we will present another enrichment algorithm which requires the formation of new basis functions based on the solution of the previous enrichment level. We call these functions online basis functions as these basis functions are computed in the online stage of computations. With the addition of the online basis functions, we can get a much faster convergence rate than the offline adaptive method.

We first define a linear functional which generalizes the residual operator in the offline adaptive method. Given a region $\Omega \subseteq D$, let $V_{\Omega}$ be the space of functions in $V_{\text{snap}}$ which are supported in $\Omega$, i.e. $V_{\Omega} = \bigoplus_{\omega_i \subseteq \Omega} V_{\text{snap}}^{(i)}$. Let $\widehat{V}_{\Omega}$ denote the divergence free subspace of $V_{\Omega}$. Define the linear functional $R_{\Omega}$ on $V_{\Omega}$ by
$$ R_{\Omega}(v) = \int_{\Omega} \kappa^{-1} v_{\text{ms}}^m \cdot v - \int_{\Omega} \mbox{div}(v) p_{\text{ms}}^m \qquad \forall v \in V_{\Omega}.$$
The norm $\|R_{\Omega}\|_{V_{\Omega}^*}$, we use, is given by either
$$ \|R_{\Omega}\|_{V_{\Omega}^*} = \sup_{v \in V_{\Omega}} \frac{R_{\Omega}(v)}{\|v\|_{H(\text{div}; \Omega; \kappa^{-1})}} \quad \text{or} \quad \|R_{\Omega}\|_{V_{\Omega}^*} = \sup_{v \in V_{\Omega}} \frac{R_{\Omega}(v)}{\|v\|_{\mathcal{L}^2(\Omega; \kappa^{-1})}}$$
depending on which spectral problem is used. Remark that if we restrict $R_{\Omega}$ on $\widehat{V}_{\Omega}$, we have
$$ R_{\Omega}(v) = \int_{\Omega} \kappa^{-1} v_{\text{ms}}^m \cdot v \qquad \forall v \in \widehat{V}_{\Omega}.$$
For the case of $\Omega = \omega_i$ for some $i$, $R_{\Omega}$ is the same as the residual operator $R_i$ in the offline adaptive method. 

Similar to the offline adaptive method, we use $m$ to indicate the enrichment level and $V_{\text{ms}}^m$ to denote the velocity space at enrichment level $m$. In the online adaptive method, we can use either spectral problem 1 or 2. However, since the online basis functions constructed in each enrichment level are divergence free, we must ensure that before any enrichment of the space, there is at least one basis function supported on each coarse grid neighborhood, which is not divergence free.

\paragraph{Online adaptive method:}
Let $m = 0$. We start by choosing an initial number of offline basis functions, $l_i$, for each coarse grid neighborhood $\omega_i$. We use the first $l_i$ eigenfunctions from each coarse grid neighborhood $\omega_i$ to form the initial velocity space $V_{\text{ms}}^0$. If the first $l_i$ eigenfunctions of a coarse grid neighborhood are all divergence free, we may artifically construct an extra basis function that is not divergence free and include it into $V_{\text{ms}}^0$. We go to step 1 below.

\begin{description}
\item[Step 1:] Find the multiscale solution. We solve for the multiscale $v_{\text{ms}}^m \in V_{\text{ms}}^m$ and $p_{\text{ms}}^m \in Q$ as in step 1 of the offline adaptive method.

\item[Step 2:] Select non-overlapping regions. We pick non-overlapping regions $\Omega_1,\, \Omega_2,\, \dots ,\, \Omega_J \subseteq D$ such that each $\Omega_j$ is a union of some coarse grid neighborhoods.

\item[Step 3:] Solve for online basis functions. For each $\Omega_j$, we solve for $\phi_j \in \widehat{V}_{\Omega_j}$ such that
$$ R_{\Omega_j}(v) = \int_{\Omega_j} \kappa^{-1} \phi_j \cdot v \qquad \forall v \in \widehat{V}_{\Omega_j},$$ 
i.e., we solve for the Riesz representation of $R_{\Omega_j}$ in $\widehat{V}_{\Omega_j}$. Those $\phi_j$'s are the new online basis functions. We update the velocity space by setting $V_{\text{ms}}^{m+1} = V_{\text{ms}}^m \oplus \mbox{span} \{ \phi_1,\, \phi_2,\, \dots ,\, \phi_J \}$.
\end{description}

Again, after Step 3, we repeat from Step 1 until the global error indicator is small or we have certain number of basis functions.

In our calculation, after obtaining the online basis function $\phi_j$ in step 3 of the method, we sometimes normalize it before computing the matrix in the finite element method. We can see that $\phi_j$ is a projection of the multiscale solution $v_{\text{ms}}^m$ on the space $\widehat{V}_{\Omega_j}$. We have $\| \phi_j \|_{\mathcal{L}^2(D; \kappa^{-1})} = \| R_{\Omega_j} \|_{\widehat{V}_{\Omega_j}^*}$. When $v_{\text{ms}}^m$ is close to the snapshot solution (the solution $v_{\text{ms}}$ of equation (\ref{MsEquation}) with $V_{\text{ms}}$ being the whole snapshot space) in the region $\Omega_j$, by the first equation of (\ref{MsEquation}), the norm of the projection $\phi_j$ will be small. Adding $\phi_j$ directly into calculation will make the matrix in the calculation close to singular.

\section{Convergence analysis} \label{section_analysis}

In this section, we will present the proofs for the convergence of both the offline and the online adaptive method. We first define some notations that will appear in the results.

We denote the maximum number of faces of a coarse grid block by $N_{\mathcal{T}}$. Let $v_{\text{snap}} \in V_{\text{snap}}$ and $p_{\text{snap}} \in Q$ denote the snapshot solution, i.e. $v_{\text{snap}}$ and $p_{\text{snap}}$ satisfy
\begin{equation}\label{MsEquation_snapshot}
\begin{aligned}
\int_{D} \kappa^{-1} v_{\text{snap}} \cdot w - \int_{D} \mbox{div}(w) p_{\text{snap}} &= 0 && \forall w \in V^{0}_{\text{snap}}, \\
\int_{D} \mbox{div}(v_{\text{snap}}) q &= \int_{D} fq && \forall q \in Q,
\end{aligned}
\end{equation}
with $v_{\text{snap}} \cdot n = g_{H}$ on $\partial D$.

With these notations, we have the following result for the error indicator.

\begin{lemma} \label{lemma_upper_bound}
We have
$$\|v_{\text{snap}} - v_{\text{ms}}\|_{\mathcal{L}^2(D; \kappa^{-1})}^2 \leq C_{\text{err}} \sum_{i=1}^{N_e} \|R_{i}\|_{{V_{\text{snap}}^{(i)}}^*}^2 (\lambda_{l_i + 1}^{(i)})^{-1}.$$
If spectral problem 1 is used, then $C_{\text{err}} = \frac{C_{V}H}{h}$ and $\|\cdot\|_{V_{\text{snap}}^{(i)}}$ is $\|\cdot\|_{H(\text{div}; \omega_i; \kappa^{-1})}$. If spectral problem 2 is used, then $C_{\text{err}} = N_{\mathcal{T}}$ and $\|\cdot\|_{V_{\text{snap}}^{(i)}}$ is $\|\cdot\|_{\mathcal{L}^2(\omega_i; \kappa^{-1})}$. The value of $C_{V}$ depends on the polynomial order of the fine grid basis functions in $V_{\text{snap}}$.
\end{lemma}

\begin{proof}
For any $v \in V_{\text{snap}}$, we can see from the construction of $V_{\text{snap}}$ that $\mbox{div}(v)$ is constant on each coarse grid block. Hence, by the second equation of (\ref{MsEquation}) and (\ref{MsEquation_snapshot}), we have
\begin{equation*}
\begin{aligned}
\int_{D} \mbox{div} (v_{\text{snap}} - v_{\text{ms}})^{2} &= \int_{D} \mbox{div} (v_{\text{snap}}) \mbox{div} (v_{\text{snap}} - v_{\text{ms}}) \\
& \quad - \int_{D} \mbox{div}(v_{\text{ms}}) \mbox{div} (v_{\text{snap}} - v_{\text{ms}})\\
&= 0.
\end{aligned}
\end{equation*}
Thus, $\mbox{div}(v_{\text{ms}} - v_{\text{snap}}) = 0$.

\noindent
Next, since $v_{\text{snap}} - v_{\text{ms}} \in V_{\text{snap}}$, we have
\begin{equation*}
\begin{aligned}
\int_D \kappa^{-1} |v_{\text{snap}} - v_{\text{ms}}|^2
&= \int_D \kappa^{-1} (v_{\text{snap}} - v_{\text{ms}}) \cdot (v_{\text{snap}} - v_{\text{ms}}) \\
& \quad - \int_D \mbox{div} (v_{\text{snap}} - v_{\text{ms}}) (p_{\text{snap}} - p_{\text{ms}}) \\
\end{aligned}
\end{equation*}
Using the first equation of (\ref{MsEquation_snapshot}), we get
\begin{equation*}
\begin{aligned}
\int_D \kappa^{-1} |v_{\text{snap}} - v_{\text{ms}}|^2
&= - \int_D \kappa^{-1} v_{\text{ms}} \cdot (v_{\text{snap}} - v_{\text{ms}}) + \int_D \mbox{div} (v_{\text{snap}} - v_{\text{ms}}) p_{\text{ms}}.
\end{aligned}
\end{equation*}
By definition, we can write
\begin{equation*}
\begin{aligned}
\int_D \kappa^{-1} |v_{\text{snap}} - v_{\text{ms}}|^2
&= -\langle R_D ,\, v_{\text{snap}} - v_{\text{ms}} \rangle.
\end{aligned}
\end{equation*}

We decompose $v_{\text{snap}} - v_{\text{ms}}$ as the sum of functions from the $V_{\text{snap}}^{(i)}$'s, i.e. $v_{\text{snap}} - v_{\text{ms}} = \sum_{E_i \in \mathcal{E}^H} v^{(i)}$ where $v^{(i)} \in V_{\text{snap}}^{(i)}$. Each $v^{(i)}$ can further be written as sum of two components: one in $V_{\text{ms}}^{(i)}$ and the other in $\mbox{span}(\beta_{\text{snap}}^{(i)} \setminus \beta_{\text{ms}}^{(i)})$. Let $v_{\text{r}}^{(i)}$ be the latter one. We get
\begin{equation*}
\begin{aligned}
\langle R_D ,\, v_{\text{snap}} - v_{\text{ms}} \rangle &= \sum_{i=1}^{N_e} \langle R_{i} ,\, v^{(i)} \rangle \\
&= \sum_{i=1}^{N_e} \langle R_{i} ,\, v_{\text{r}}^{(i)} \rangle. \\
\end{aligned}
\end{equation*}
Using the definition of the spectral problems, we get
\begin{equation*}
\begin{aligned}
 \sum_{i=1}^{N_e} \langle R_{i} ,\, v_{\text{r}}^{(i)} \rangle
&\leq \sum_{i=1}^{N_e} \|R_{i}\|_{({V_{\text{snap}}^{(i)}})^*} \|v_{\text{r}}^{(i)}\|_{V_{\Omega_j}} \\
&\leq \sum_{i=1}^{N_e} \|R_{i}\|_{({V_{\text{snap}}^{(i)}})^*} (s_i (v_r^{(i)},\, v_r^{(i)}))^{\frac{1}{2}},
\end{aligned}
\end{equation*}
where $\|\cdot\|_{V_{\omega_i}} = \|\cdot\|_{H(\text{div}; \omega_i; \kappa^{-1})}$ if spectral problem 1 is used, and $\|\cdot\|_{V_{\omega_i}} = \|\cdot\|_{\mathcal{L}^2(\omega_i; \kappa^{-1})}$ if spectral problem 2 is used.

Next, we consider the two spectral problems separately.

\paragraph{Spectral problem 1:} For each $i$, we have
\begin{equation*}
\begin{aligned}
s_i(v_{\text{r}}^{(i)}, v_{\text{r}}^{(i)})
&\leq H (\lambda_{l_i+1}^{(i)})^{-1} a_i(v_{\text{r}}^{(i)},\, v_{\text{r}}^{(i)}) \\
&\leq H (\lambda_{l_i+1}^{(i)})^{-1}  \int_{E_i} \kappa^{-1} ((v_{\text{snap}} - v_{\text{ms}}) \cdot n_i)^2. \\
\end{aligned}
\end{equation*}
Thus, by Cauchy Schwarz inequality,

\begin{equation*}
\begin{aligned}
\langle R_D ,\, v_{\text{snap}} - v_{\text{ms}} \rangle 
&\leq \sqrt{H} \sum_{i=1}^{N_e} \|R_{i}\|_{({V_{\text{snap}}^{(i)}})^*} (\lambda_{l_i+1}^{(i)})^{-\frac{1}{2}} \left(\int_{E_i} \kappa^{-1} ((v_{\text{snap}} - v_{\text{ms}}) \cdot n_i)^2 \right)^{\frac{1}{2}}\\
&\leq \sqrt{H} \left( \sum_{i=1}^{N_e} \|R_{i}\|_{({V_{\text{snap}}^{(i)}})^*}^2 (\lambda_{l_i+1}^{(i)})^{-1} \right)^{\frac{1}{2}} \left( \sum_{i=1}^{N_e} \int_{E_i} \kappa^{-1} ((v_{\text{snap}} - v_{\text{ms}}) \cdot n_i)^2 \right)^{\frac{1}{2}} \\
&\leq \sqrt{\frac{C_{V}H}{h}} \left( \sum_{i=1}^{N_e} \|R_{i}\|_{(V_{\text{snap}}^{(i)})^*}^2 (\lambda_{l_i+1}^{(i)})^{-1} \right)^{\frac{1}{2}} \|v_{\text{snap}} - v_{\text{ms}}\|_{\mathcal{L}^2(D; \kappa^{-1})},
\end{aligned}
\end{equation*}
where 
$C_{V}$ is a constant depending on the polynomial order of the fine grid basis functions in $V_{\text{snap}}$.

\paragraph{Spectral problem 2:} Similar to spectral problem 1, for each $j$, we have
\begin{equation*}
\begin{aligned}
s_i(v_{\text{r}}^{(i)}, v_{\text{r}}^{(i)})
&\leq (\lambda_{l_i+1}^{(i)})^{-1} \int_{\omega_i} \kappa^{-1} |\widetilde{v^{(i)}}|^2  \\
&\leq (\lambda_{l_i+1}^{(i)})^{-1} \int_{\omega_i} \kappa^{-1} |v_{\text{snap}} - v_{\text{ms}}|^2, \\
\end{aligned}
\end{equation*}
where we used the minimum energy property of $\widetilde{v^{(i)}}$. 
Therefore,
\begin{equation*}
\begin{aligned}
\langle R_D ,\, v_{\text{snap}} - v_{\text{ms}} \rangle
&\leq \left(\sum_{i=1}^{N_e} \|R_{i}\|_{(V_{\text{snap}}^{(i)})^*}^2 (\lambda_{l_i+1}^{(i)})^{-1} \right)^{\frac{1}{2}} \left( \sum_{i=1}^{N_e} \int_{\omega_i} \kappa^{-1} |v_{\text{snap}} - v_{\text{ms}}|^2 \right)^{\frac{1}{2}} \\
&\leq \sqrt{N_{\mathcal{T}}} \left(\sum_{i=1}^{N_e} \|R_{\Omega_j}\|_{(V_{\text{snap}}^{(i)})^*}^2 (\lambda_{l_i+1}^{(i)})^{-1} \right)^{\frac{1}{2}} \|v_{\text{snap}} - v_{\text{ms}}\|_{\mathcal{L}^2(D; \kappa^{-1})}.
\end{aligned}
\end{equation*}

\end{proof} 

From the proof of Lemma \ref{lemma_upper_bound}, the error between the multiscale solution and the snapshot solution is bounded above by the norm of the global residual operator $R_D$, which in turn can be estimated by the sum of the error indicator $\|R_i\|_{(V_{\text{snap}}^{(i)})^*} (\lambda_{l_i + 1}^{(i)})^{-1}$. 

Next, before we show the convergence of the offline adaptive method, we need a local version of the inf-sup condition in \cite{chung2014mixed}. We will show the proof of this simplier case and compute the constant in the result. Here is the statement.

\begin{lemma} \label{lemma_infsup}
For coarse grid neighborhood $\omega_i$, write $\omega_i = K_1 \cup K_2$ where $K_1$ and $K_2$ are the two coarse grid blocks composing $\omega_i$. Then, for any $p \in Q$, we have
$$ \|p - \overline{p}\|_{\mathcal{L}^2(\omega_i)} = C_{\text{sup}}^{i} \sup_{v \in V_{\text{ms}}^{(i)}} \frac{ \int_{\omega_i} \text{div}(v) p }{ \|v\|_{\mathcal{L}^2(\omega_i; \kappa^{-1})} }, $$
where $\overline{p} = \frac{1}{|\omega_i|} \int_{\omega_i} p$ and $C_{\text{sup}}^i$ is the infimum of $\sqrt{\frac{|K_1||K_2|}{|K_1| + |K_2|}}\|v\|_{\mathcal{L}^2(\omega_i; \kappa^{-1})}$ over all $v \in V_{\text{ms}}^{(i)}$ with  $\int_{E_i} v \cdot n_i = 1$.
\end{lemma}

\begin{proof}
Note that
$$ \int_{\omega_i} \text{div}(v) \overline{p} = \overline{p} \int_{\partial \omega_i} v \cdot n = 0. $$
We may assume $\overline{p} = 0$. Let $p_0 = p / \|p\|_{\mathcal{L}^2(\omega_i)}$. Using this notation, we have
\begin{equation*}
\begin{aligned}
\int_{\omega_i} \text{div}(v) p &= \|p\|_{\mathcal{L}^2(\omega_i)} \left( p_0 |_{K_1} \int_{K_1} \text{div}(v) + p_0 |_{K_2} \int_{K_2} \text{div}(v) \right) \\
&=  \|p\|_{\mathcal{L}^2(\omega_i)} ( p_0 |_{K_1} - p_0 |_{K_2}) \int_{E_i} v \cdot n_i
\end{aligned}
\end{equation*}
Finally, we evaluate $(p_0 |_{K_1} - p_0 |_{K_2})$. Since $\overline{p}_0 = 0$, we have
$$ |K_1| (p_0 |_{K_1}) + |K_2| (p_0 |_{K_2}) = 0.$$
Using $\|p_0\|_{\mathcal{L}^2(\omega_i)} = 1$, we get
$$ |K_1| (p_0 |_{K_1})^2 + |K_2| (p_0 |_{K_2})^2 = 1.$$
Using these two, one can check that
$$p_0 |_{K_1} - p_0 |_{K_2} = \pm \sqrt{\frac{1}{|K_1|} + \frac{1}{|K_2|}}.$$
Hence, we have
$$ \sup_{v \in V_{\text{ms}}^{(i)}} \frac{ \int_{\omega_i} \text{div}(v) p }{ \|v\|_{\mathcal{L}^2(\omega_i; \kappa^{-1})} }
=  \|p\|_{\mathcal{L}^2(\omega_i)} \sqrt{\frac{1}{|K_1|} + \frac{1}{|K_2|}} \sup_{v \in V_{\text{ms}}^{(i)}} \frac{\int_{E_i} v \cdot n_i}{\|v\|_{\mathcal{L}^2(\omega_i; \kappa^{-1})}}, $$
which completes the proof.
\end{proof}

We define some symbols before we move on to the proof of the convergence. Let $R_i^m$ denote the residual operator $R_i$ using the solution $(v_{\text{ms}}^m,\, p_{\text{ms}}^m)$. We define
\begin{equation}
\label{eq:Si}
S_i^m = \|R_i^m\|_{(V_{\text{snap}}^{(i)})^*} (\lambda_{l_i^{m+1}}^{(i)})^{-\frac{1}{2}}.
\end{equation}
This symbol $S_i^m$ is indeed the error indicator $\eta_i$ at the $m$-th enrichment level. We have the following lemma for this symbol.

\begin{lemma} \label{lemma_offline}
Let $S_i^m$ be the expression defined in (\ref{eq:Si}). Then, for any $\alpha > 0$, we have 
$$ (S_i^{m+1})^2 \leq (1+\alpha) \frac{\lambda_{l_i^m+1}^{(i)}}{\lambda_{l_i^{m+1}+1}^{(i)}} (S_i^m)^2 + (1+\alpha^{-1}) D_m^i \| v_{\text{ms}}^{m+1} - v_{\text{ms}}^m \|_{\mathcal{L}^2(\omega_i; \kappa^{-1})}^2,$$
where $D_m^i = 2 (\lambda_{l_i^{m+1}+1}^{(i)})^{-1} (\max \{C_{\text{sup}}^{i, m}, 1  \})^2$ and $C_{\text{sup}}^{i, m}$ is the constant from Lemma \ref{lemma_infsup} at the $m$-th enrichment level.
\end{lemma}

\begin{proof}
For any $v \in V_{\text{snap}}^{(i)}$, using the definition of $R_i^m$, we have
\begin{equation*}
\begin{aligned}
R_i^{m+1} (v) 
&= R_i^m(v) + \int_{\omega_i} \kappa^{-1} (v_{\text{ms}}^{m+1} - v_{\text{ms}}^{m}) \cdot v - \int_{\omega_i} \text{div}(v) (p_{\text{ms}}^{m+1} - p_{\text{ms}}^{m}). \\
\end{aligned}
\end{equation*}
Taking supremum with respect to $v$ and noting that $V_{\text{snap}}^{(i+1)}=V_{\text{snap}}^{(i)}$, we get
$$ S_i^{m+1} \leq \left( \frac{\lambda_{l_i^m+1}^{(i)}}{\lambda_{l_i^{m+1}+1}^{(i)}} \right)^{\frac{1}{2}} S_i^m + I,$$
where
$$I = (\lambda_{l_i^{m+1}+1}^{(i)})^{\frac{1}{2}} \sup_{v\in V_{\text{snap}}^{(i)}} \frac{ \int_{\omega_i} \kappa^{-1} (v_{\text{ms}}^{m+1} - v_{\text{ms}}^{m}) \cdot v - \int_{\omega_i} \text{div}(v) (p_{\text{ms}}^{m+1} - p_{\text{ms}}^{m}) }{ \|v\|_{H(\text{div}; \omega_i; \kappa^{-1})} }.$$
Next, we have 
\begin{equation*}
\begin{aligned}
& \quad \int_{\omega_i} \kappa^{-1} (v_{\text{ms}}^{m+1} - v_{\text{ms}}^{m}) \cdot v - \int_{\omega_i} \text{div}(v) (p_{\text{ms}}^{m+1} - p_{\text{ms}}^{m}) \\
&= \int_{\omega_i} \kappa^{-1} (v_{\text{ms}}^{m+1} - v_{\text{ms}}^{m}) \cdot v - \int_{\omega_i} \text{div}(v) (p_{\text{ms}}^{m+1} - p_{\text{ms}}^{m} - \overline{(p_{\text{ms}}^{m+1} - p_{\text{ms}}^{m})}) \\
&\leq \| v_{\text{ms}}^{m+1} - v_{\text{ms}}^m \|_{\mathcal{L}^2(\omega_i; \kappa^{-1})} \|v\|_{\mathcal{L}^2(\omega_i; \kappa^{-1})} + \|\text{div}(v)\|_{\mathcal{L}^2(\omega_i)} \|p_{\text{ms}}^{m+1} - p_{\text{ms}}^m - \overline{(p_{\text{ms}}^{m+1} - p_{\text{ms}}^{m})}\|_{\mathcal{L}^2(\omega_i)}
\end{aligned}
\end{equation*}
where $\overline{(p_{\text{ms}}^{m+1} - p_{\text{ms}}^{m})}$ is the average value of $p_{\text{ms}}^{m+1} - p_{\text{ms}}^{m}$ over $\omega_i$.
Using Lemma \ref{lemma_infsup}, we get
\begin{equation*}
\begin{aligned}
 \|p_{\text{ms}}^{m+1} - p_{\text{ms}}^m - \overline{(p_{\text{ms}}^{m+1} - p_{\text{ms}}^{m})}\|_{\mathcal{L}^2(\omega_i)} 
&\leq C_{\text{sup}}^{i, m} \sup_{v \in V_{\text{ms}}^{(i), m}} \frac{ \int_D \text{div} (v) (p_{\text{ms}}^{m+1} - p_{\text{ms}}^m) }{ \|v\|_{\mathcal{L}^2(\omega_i; \kappa^{-1})} }\\
&\leq  C_{\text{sup}}^{i, m} \sup_{v \in V_{\text{ms}}^{(i), m}} \frac{ \int_D \kappa^{-1} (v_{\text{ms}}^{m+1} - v_{\text{ms}}^m) \cdot v }{ \|v\|_{\mathcal{L}^2(\omega_i; \kappa^{-1})} }\\
&\leq C_{\text{sup}}^{i, m} \| v_{\text{ms}}^{m+1} - v_{\text{ms}}^m \|_{\mathcal{L}^2(\omega_i; \kappa^{-1})},
\end{aligned}
\end{equation*}
where $C_{\text{sup}}^{i, m}$ and $V_{\text{ms}}^{(i), m}$ denote the constant $C_{\text{sup}}^i$ in Lemma \ref{lemma_infsup} and the space $V_{\text{ms}}^{(i)}$ at enrichment level $m$. Hence, we estimate $I$ as
$$ I \leq (\lambda_{l_i^{m+1}+1}^{(i)})^{-\frac{1}{2}} \sqrt{2} \max \{C_{\text{sup}}^{i, m}, 1  \} \| v_{\text{ms}}^{m+1} - v_{\text{ms}}^m \|_{\mathcal{L}^2(\omega_i; \kappa^{-1})}. $$
Therefore we have
$$ S_i^{m+1} \leq \left( \frac{\lambda_{l_i^m+1}^{(i)}}{\lambda_{l_i^{m+1}+1}^{(i)}} \right)^{\frac{1}{2}} S_i^m + (\lambda_{l_i^{m+1}+1}^{(i)})^{-\frac{1}{2}} \sqrt{2} \max \{C_{\text{sup}}^{i, m}, 1  \} \| v_{\text{ms}}^{m+1} - v_{\text{ms}}^m \|_{\mathcal{L}^2(\omega_i; \kappa^{-1})}. $$
And so, we get
$$ (S_i^{m+1})^2 \leq (1+\alpha) \frac{\lambda_{l_i^m+1}^{(i)}}{\lambda_{l_i^{m+1}+1}^{(i)}} (S_i^m)^2 + (1+\alpha^{-1}) D_m^i \| v_{\text{ms}}^{m+1} - v_{\text{ms}}^m \|_{\mathcal{L}^2(\omega_i; \kappa^{-1})}^2,$$
where $D_m^i = 2 (\lambda_{l_i^{m+1}+1}^{(i)})^{-1} (\max \{C_{\text{sup}}^{i, m}, 1  \})^2$.
\end{proof}

Using this lemma, we have the following result for the convergence of the offline adaptive method.

\begin{theorem} \label{theorem_offline}
Using the notations in the offline adaptive method, there exist positive constants $\delta_0$, $\rho$ and a decreasing sequence of positive numbers $\{L_j\}$ such that the following contracting property holds
$$  \|v_{\text{snap}} - v_{\text{ms}}^{m+1}\|_{\mathcal{L}^2(D; \kappa^{-1})}^2 + \frac{1}{L_j} \sum_{i = 1}^{N_e} (S_i^{m+1})^2 \leq \epsilon_j \left( \|v_{\text{snap}} - v_{\text{ms}}^{m}\|_{\mathcal{L}^2(D; \kappa^{-1})}^2 + \frac{1}{L_j} \sum_{i = 1}^{N_e} (S_i^m)^2  \right),$$
for any $j \leq m$, where $\delta_0$ and $\rho$ satisfy
$$ \frac{\lambda_{l_i^m+1}^{(i)}}{\lambda_{l_i^{m+1}+1}^{(i)}} \leq \delta_0 < 1 - (1 - \delta_0)\theta^2 < \rho < 1,$$
for any coarse grid neighborhood $\omega_i$ that are selected to add basis functions, and 
$$ \epsilon_j = \frac{C_{\text{err}} L_j + \rho}{C_{\text{err}} L_j + 1}.$$
The definition of $\{L_j\}$ is given by (\ref{definition_L}).
\end{theorem}

\begin{proof}
In the offline adaptive method, we fixed $0 < \theta < 1$ and we choose an index set $I$ such that
$$ \theta^2 \sum_{i = 1}^{N_e} \eta_i^2 \leq \sum_{i \in I} \eta_i^2$$
We write
$$ \sum_{i = 1}^{N_e} (S_i^{m+1})^2 = \sum_{i \in I} (S_i^{m+1})^2 + \sum_{i \not\in I} (S_i^{m+1})^2.$$
Using Lemma \ref{lemma_offline}, we have
\begin{equation*}
\begin{aligned}
\sum_{i = 1}^{N_e} (S_i^{m+1})^2 &\leq \sum_{i \in I} \left( (1+\alpha) \frac{ \lambda_{l_i^m+1}^{(i)} }{ \lambda_{l_i^{m+1}+1}^{(i)} } (S_i^m)^2 + (1 + \alpha^{-1}) D_m^i \|v_{\text{ms}}^{m+1} - v_{\text{ms}}^m\|^2_{\mathcal{L}^2(\omega_i; \kappa^{-1})} \right) \\
& \quad + \sum_{i \not\in I} \left( (1+\alpha) (S_i^m)^2 + (1 + \alpha^{-1}) D_m^i \|v_{\text{ms}}^{m+1} - v_{\text{ms}}^m\|^2_{\mathcal{L}^2(\omega_i; \kappa^{-1})} \right)
\end{aligned}
\end{equation*}
We define
\begin{equation} \label{definition_L}
 L_{m} = N_{\mathcal{T}}(1+\alpha^{-1}) \max_{E_i \in \mathcal{E}^H} D_m^i,
\end{equation}
and we assume the number of additional offline basis functions in each enrichment level is chosen such that
$$ \max_{i \in I} \frac{ \lambda_{l_i^m+1}^{(i)} }{ \lambda_{l_i^{m+1}+1}^{(i)} } \leq \delta_0 < 1,$$
where $\delta_0$ is a fixed constant.
Using this, we get
$$ \sum_{i = 1}^{N_e} (S_i^{m+1})^2 \leq (1 + \alpha) \sum_{i = 1}^{N_e} (S_i^m)^2 - (1 + \alpha)(1 - \delta_0) \theta^2 \sum_{i = 1}^{N_e} (S_i^m)^2 + L_{m} \|v_{\text{ms}}^{m+1} - v_{\text{ms}}^m\|^2_{\mathcal{L}^2(D; \kappa^{-1})}.$$
We let $\rho = (1 + \alpha)(1 - (1 - \delta_0)\theta^2)$ and take $\alpha$ small enough so that $0 < \rho < 1$. Observed that $\{L_m\}$ is a decreasing sequence, so we may take any $j \leq m$. We now have
\begin{equation} \label{off_conv_1}
\sum_{i = 1}^{N_e} (S_i^{m+1})^2 \leq \rho \sum_{i = 1}^{N_e} (S_i^m)^2 + L_j \|v_{\text{ms}}^{m+1} - v_{\text{ms}}^m\|_{\mathcal{L}^2(D; \kappa^{-1})}.
\end{equation}
Note that $\text{div}(v_{\text{ms}}^{m+1} - v_{\text{ms}}^{m}) = 0$ and $v_{\text{ms}}^{m+1} - v_{\text{ms}}^{m} \in V_{\text{ms}}^{m+1}$. Therefore, by the first equation of (\ref{MsEquation_snapshot}) and (\ref{MsEquation_offline}), we get
$$ \int_D \kappa^{-1} (v_{\text{snap}} - v_{\text{ms}}^{m+1}) \cdot (v_{\text{ms}}^{m+1} - v_{\text{ms}}^{m}) = 0,$$
and so
$$\|v_{\text{snap}} - v_{\text{ms}}^{m}\|_{\mathcal{L}^2(D; \kappa^{-1})}^2
= \|v_{\text{snap}} - v_{\text{ms}}^{m+1}\|_{\mathcal{L}^2(D; \kappa^{-1})}^2 + \|v_{\text{ms}}^{m+1} - v_{\text{ms}}^{m}\|_{\mathcal{L}^2(D; \kappa^{-1})}^2, $$
which means
$$\|v_{\text{ms}}^{m+1} - v_{\text{ms}}^{m}\|_{\mathcal{L}^2(D; \kappa^{-1})}^2
= \|v_{\text{snap}} - v_{\text{ms}}^{m}\|_{\mathcal{L}^2(D; \kappa^{-1})}^2 - \|v_{\text{snap}} - v_{\text{ms}}^{m+1}\|_{\mathcal{L}^2(D; \kappa^{-1})}^2.$$
Putting this into (\ref{off_conv_1}), we get
$$  \|v_{\text{snap}} - v_{\text{ms}}^{m+1}\|_{\mathcal{L}^2(D; \kappa^{-1})}^2 + \frac{1}{L_j} \sum_{i = 1}^{N_e} (S_i^{m+1})^2 \leq \|v_{\text{snap}} - v_{\text{ms}}^{m}\|_{\mathcal{L}^2(D; \kappa^{-1})}^2 + \frac{\rho}{L_j}  \sum_{i = 1}^{N_e} (S_i^m)^2. $$
From Lemma \ref{lemma_upper_bound}, we have
$$ \|v_{\text{snap}} - v_{\text{ms}}^m\|_{\mathcal{L}^2(D; \kappa^{-1})} \leq C_{\text{err}} \sum_{i = 1}^{N_e} (S_i^m)^2. $$
Hence, 
$$  \|v_{\text{snap}} - v_{\text{ms}}^{m+1}\|_{\mathcal{L}^2(D; \kappa^{-1})}^2 + \frac{1}{L_j} \sum_{i = 1}^{N_e} (S_i^{m+1})^2 \leq (1 - \beta) \|v_{\text{snap}} - v_{\text{ms}}^{m}\|_{\mathcal{L}^2(D; \kappa^{-1})}^2 + (\beta C_{\text{err}} + \frac{\rho}{L_j})  \sum_{i = 1}^{N_e} (S_i^m)^2. $$
Finally, we take $\beta = \frac{1 - \rho}{1 + C_{\text{err}} L_j}$ to get
$$  \|v_{\text{snap}} - v_{\text{ms}}^{m+1}\|_{\mathcal{L}^2(D; \kappa^{-1})}^2 + \frac{1}{L_j} \sum_{i = 1}^{N_e} (S_i^{m+1})^2 \leq (1 - \beta) \|v_{\text{snap}} - v_{\text{ms}}^{m}\|_{\mathcal{L}^2(D; \kappa^{-1})}^2 + \frac{1-\beta}{L_j}  \sum_{i = 1}^{N_e} (S_i^m)^2, $$
which completes the proof.
\end{proof}

From Theorem \ref{theorem_offline}, we can see that the convergence rate depends on the two constants $\theta$ and $\delta_0$, which are fixed before we carry out the enrichment algorithm. The constant $\theta$ controls the number of coarse grid neighborhoods, where enrichment is needed. And the constant $\delta_0$ is related to the number of basis functions we have to add in each coarse grid neighborhood. Note that we have the following inequality for the convergence rate
$$ \epsilon_j > 1 - \frac{(1-\delta_0)\theta^2}{C_{\text{err}}L_j + 1}.$$
Thus, in order to have a fast convergence, we need a small $\delta_0$ and a large $\theta$, which means there is a tradeoff between the convergence rate and the number of basis functions used.

Now, we state the convergence result for the online adaptive method.

\begin{theorem} \label{theorem_online}
Using the notations in the online adaptive methods. We have
$$ \|v_{\text{snap}} - v_{\text{ms}}^{m+1}\|_{\mathcal{L}^2(D; \kappa^{-1})}^2 \leq \left( 1 - \frac{\sum_{j=1}^J \|R_{\Omega_j}\|_{\widehat{V}_{\Omega_j}^*}^2}{C_{\text{err}} \sum_{i=1}^{N_e}  \|R_{i}\|_{V_{(\text{snap}}^{(i)})^*}^2 (\lambda_{l_i + 1}^{(i)})^{-1}} \right) \|v_{\text{snap}} - v_{\text{ms}}^m\|_{\mathcal{L}^2(D; \kappa^{-1})}^2.$$
\end{theorem}

\begin{proof}
For any $v \in V_{\text{ms}}^{m+1}$ with $\mbox{div}(v) = 0$, we have
\begin{equation} 
\begin{aligned}
\|v_{\text{snap}} - v_{\text{ms}}^{m+1} + v\|_{\mathcal{L}^2(D; \kappa^{-1})}^2 &= \|v_{\text{snap}} - v_{\text{ms}}^{m+1}\|_{\mathcal{L}^2(D; \kappa^{-1})}^2 + \|v\|_{\mathcal{L}^2(D; \kappa^{-1})}^2 \\
& \quad + 2\int_D \kappa^{-1} (v_{\text{snap}} - v_{\text{ms}}^{m+1}) \cdot v \\
&= \|v_{\text{snap}} - v_{\text{ms}}^{m+1}\|_{\mathcal{L}^2(D; \kappa^{-1})}^2 + \|v\|_{\mathcal{L}^2(D; \kappa^{-1})}^2 \\
&\geq \|v_{\text{snap}} - v_{\text{ms}}^{m+1}\|_{\mathcal{L}^2(D; \kappa^{-1})}^2 \label{step:prop_online_1}
\end{aligned}
\end{equation}

\noindent
Let the new online basis functions $\phi_1,\, \phi_2 ,\, \dots ,\, \phi_J$ be normalized such that $\|\phi_j\|_{\mathcal{L}^2(\Omega_j; \kappa^{-1})} = 1$. Let $v = v_{\text{ms}}^{m+1} - v_{\text{ms}}^m + \alpha_1 \phi_1 + \dots + \alpha_J \phi_J$ where $\alpha_j = \int_{\Omega_j} \kappa^{-1} v_{\text{ms}}^m \cdot \phi_j = \|R_{\Omega_j}\|_{\widehat{V}_{\Omega_j}^*}$. Check that $v$ is divergence free. Using (\ref{step:prop_online_1}),
\begin{equation*}
\begin{aligned}
\|v_{\text{snap}} - v_{\text{ms}}^{m+1}\|_{\mathcal{L}^2(D; \kappa^{-1})}^2 &\leq \|v_{\text{snap}} - v_{\text{ms}}^m + \alpha_1 \phi_1 + \dots + \alpha_J \phi_J\|_{\mathcal{L}^2(D; \kappa^{-1})}^2 \\
&= \|v_{\text{snap}} - v_{\text{ms}}^m\|_{\mathcal{L}^2(D; \kappa^{-1})}^2 + \|\alpha_1 \phi_1 + \dots + \alpha_J \phi_J\|_{\mathcal{L}^2(D; \kappa^{-1})}^2 \\
& \quad + 2\int_D \kappa^{-1} (v_{\text{snap}} - v_{\text{ms}}^m) \cdot (\alpha_1 \phi_1 + \cdots + \alpha_J \phi_J) \\
&= \|v_{\text{snap}} - v_{\text{ms}}^m\|_{\mathcal{L}^2(D; \kappa^{-1})}^2 + \|\alpha_1 \phi_1 + \dots + \alpha_J \phi_J\|_{\mathcal{L}^2(D; \kappa^{-1})}^2 \\
& \quad - 2\int_D \kappa^{-1} v_{\text{ms}}^m \cdot (\alpha_1 \phi_1 + \cdots + \alpha_J \phi_J) \\
\end{aligned}
\end{equation*}

\noindent
Recall that $\Omega_1,\, \dots ,\, \Omega_J$ are non-overlapping and each $\phi_j$ is supported on $\Omega_j$. Therefore,
\begin{equation*}
\begin{aligned}
& \|\alpha_1 \phi_1 + \dots + \alpha_J \phi_J\|_{\mathcal{L}^2(D; \kappa^{-1})}^2  - 2\int_D \kappa^{-1} v_{\text{ms}}^m \cdot (\alpha_1 \phi_1 + \cdots + \alpha_J \phi_J) \\
&= -(\|R_{\Omega_1}\|_{\widehat{V}_{\Omega_1}^*}^2 + \cdots + \|R_{\Omega_J}\|_{\widehat{V}_{\Omega_J}^*}^2).
\end{aligned}
\end{equation*}
Hence, we have
\begin{equation*}
\begin{aligned}
\|v_{\text{snap}} - v_{\text{ms}}^{m+1}\|_{\mathcal{L}^2(D; \kappa^{-1})}^2 &\leq \|v_{\text{snap}} - v_{\text{ms}}^m\|_{\mathcal{L}^2(D; \kappa^{-1})}^2 - \sum_{j=1}^J \|R_{\Omega_j}\|_{\widehat{V}_{\Omega_j}^*}^2 \\
&= \left( 1 - \frac{\sum_{j=1}^J \|R_{\Omega_j}\|_{\widehat{V}_{\Omega_j}^*}^2}{\|v_{\text{snap}} - v_{\text{ms}}^m\|_{\mathcal{L}^2(D; \kappa^{-1})}^2} \right) \|v_{\text{snap}} - v_{\text{ms}}^m\|_{\mathcal{L}^2(D; \kappa^{-1})}^2 \\
&= \left( 1 - \frac{\sum_{j=1}^J \|R_{\Omega_j}\|_{\widehat{V}_{\Omega_j}^*}^2}{C_{\text{err}} \sum_{i=1}^{N_e}  \|R_{i}\|_{(V_{\text{snap}}^{(i)})^*}^2 (\lambda_{l_i+1}^{(i)})^{-1}} \right) \|v_{\text{snap}} - v_{\text{ms}}^m\|_{\mathcal{L}^2(D; \kappa^{-1})}^2,
\end{aligned}
\end{equation*}
by Lemma \ref{lemma_upper_bound}. Hence the proof is complete. 
\end{proof}

In Theorem \ref{theorem_online}, we can see that the convergence rate of the online adaptive method can be small if the term $(\Lambda_{\min}^j)^{-1}$ is large. This term is determined at the beginning of the method when the initial number of basis functions $l_i$ is fixed. Therefore, we should choose $l_i$ so that $\lambda_{l_i+1}^{(i)}$ is significantly large. We will demonstrate the effect of choosing different initial number of basis functions in the next section.

\section{Numerical Results} \label{section_numerical}

In this section, we will present some examples using both the offline and the online adaptive methods. To test the efficiency of the methods, we will compare the error of the solution from the adaptive methods with the error of the solution obtained by uniform enrichment of the space, i.e. we increase the number of basis functions in all the coarse grid neighborhoods uniformly. We will also compare the error of the solution using the online adaptive method with different initial number of basis functions on each coarse grid neighborhood and different contrast in the permeability field. From the comparison in the online adaptive method, we can see that the convergence rate depends on the initial number of offline bases used. And if the initial basis is chosen in the appropriate way, the convergence of online adaptive method will be independent of the contrast. In the examples, we will use the following permeability fields with background value one.

\begin{figure}[H]
    \centering
\hfill
    \begin{subfigure}[t]{0.39\textwidth}
        \includegraphics[width=\textwidth]{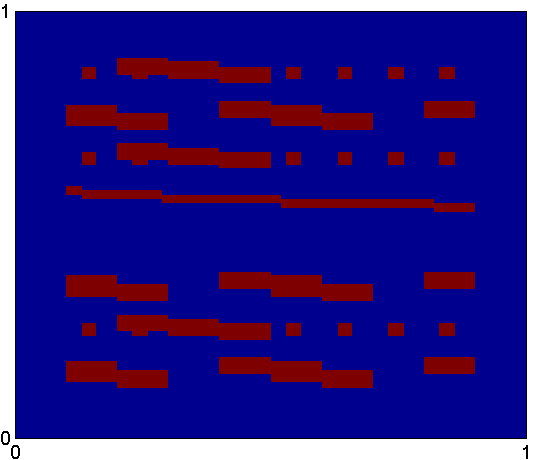}
        \caption{$\kappa_{1}$} \label{fig:permeability_field_1}
    \end{subfigure}
\hfill
    \begin{subfigure}[t]{0.39\textwidth}
        \includegraphics[width=\textwidth]{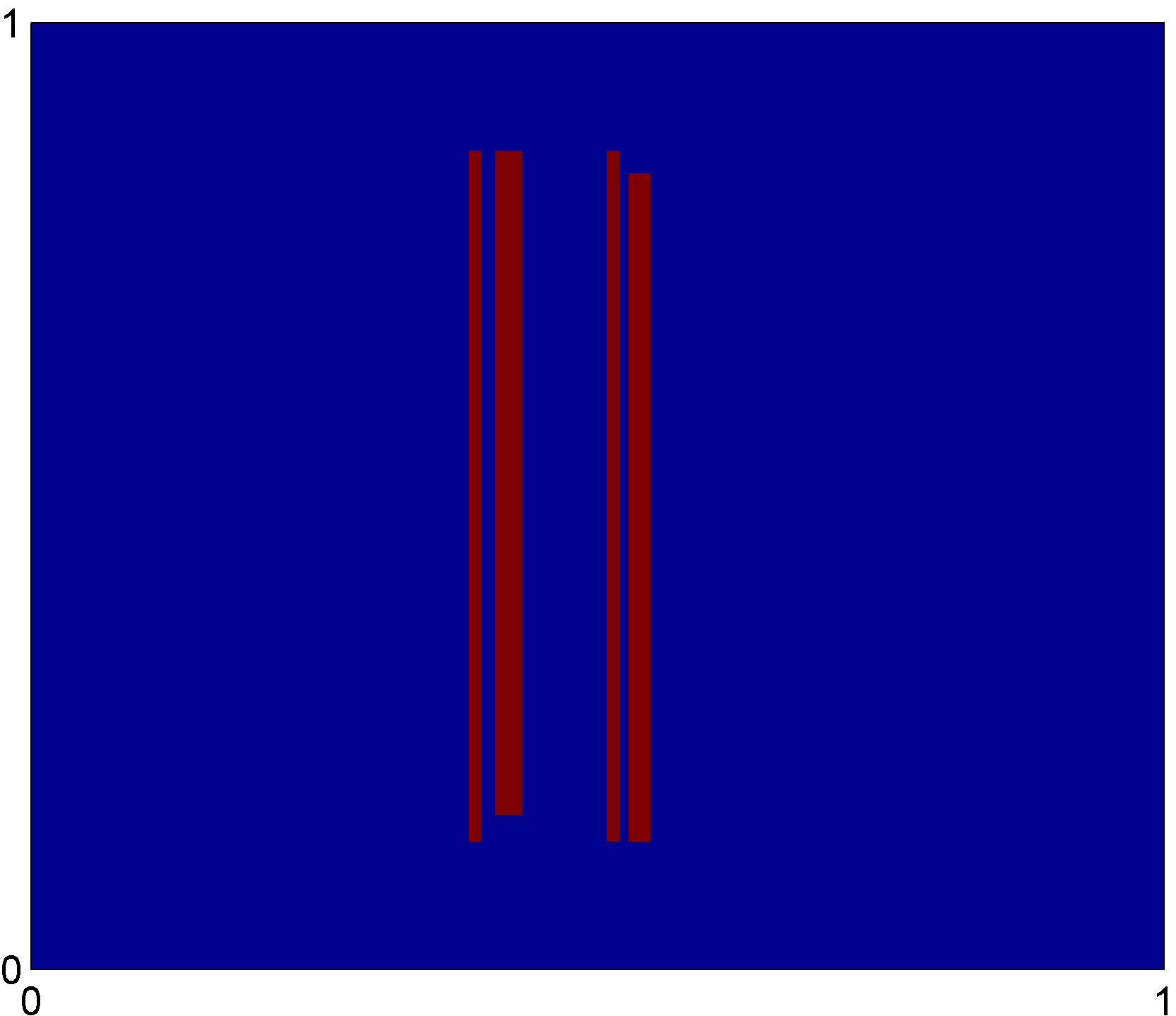}
        \caption{$\kappa_{2}$} \label{fig:permeability_field_2}
    \end{subfigure}
\hfill
\null
    \caption{Permeability fields with high contrasts (denoted in red).}\label{fig:permeability_fields}
\end{figure}

Figure \ref{fig:permeability_fields} shows the permeability fields with background value one (shown in blue) and high contrasts (shown in red). In the numerical examples, we will vary the contrast so as to see the convergence rate of the adaptive methods with different contrast values. We will use the follow snapshot error to indicate the accuracy of the methods.
\begin{equation*}
e = \frac{\| u_{\text{snap}} - u_{\text{ms}} \|_{\mathcal{L}^{2}(D;\kappa^{-1})}}{\| u_{\text{snap}} \|_{\mathcal{L}^{2}(D;\kappa^{-1})}}
\end{equation*}

\subsection{Comparing the adaptive methods with uniform enrichment}
We compare the efficiency of the adaptive methods in this example. Consider equation (\ref{equation}) on the domain $[0,\,1]^{2}$ with homogeneous boundary condition, i.e. $g = 0$. We use coarse grid size $15 \times 15$ and fine grid size $40 \times 40$ on each coarse grid. We use the permeability field $\kappa_{1}$ with contrast values 1e4 and 1e-4. The source function $f$ is set to be $1$ on top left coarse grid block, $-1$ on bottom right coarse grid block and zero elsewhere. We form the snapshot basis using spectral problem 1.

We solve the equation in the following ways.

\begin{description}
\item[Offline adaptive method:] We use the offline adaptive method with initial number of bases per coarse grid neighborhood equal to $2$, $\theta = 0.2$ and $\delta_0 = 0.5$.
\item[Online adaptive method:] We use the online adaptive method with initial basis functions obtained from spectral problem 1. The regions $\Omega_1, \dots ,\, \Omega_J$ are
\begin{enumerate}[(a)]
\item non-overlapping coarse grid neighborhoods; and
\item non-overlapping $2 \times 2$ coarse grid blocks.
\end{enumerate}
\item[Uniform enrichment:] we solve the equation again by increasing the number of bases in each coarse grid neighborhood uniformly from 2 to 40.
\end{description}

The results are compared by plotting the error $e$ against the total number of bases used (Figure \ref{fig:snapshot_error_plot_example_1}).

\begin{figure}[!h]
    \centering
    \begin{subfigure}[t]{0.49\textwidth}
        \includegraphics[width=\textwidth]{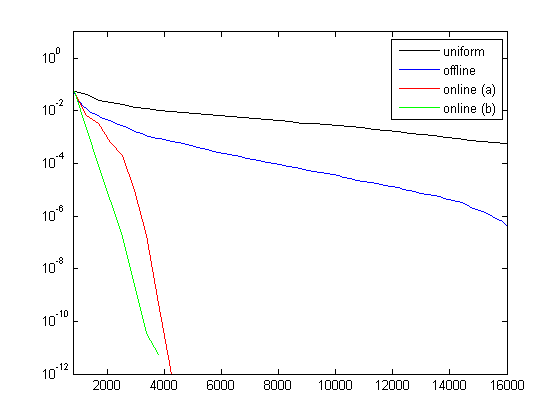}
        \caption{contrast = 1e4} \label{fig:snapshotl_error_plot_example_1_1e4}
    \end{subfigure}
    \begin{subfigure}[t]{0.49\textwidth}
        \includegraphics[width=\textwidth]{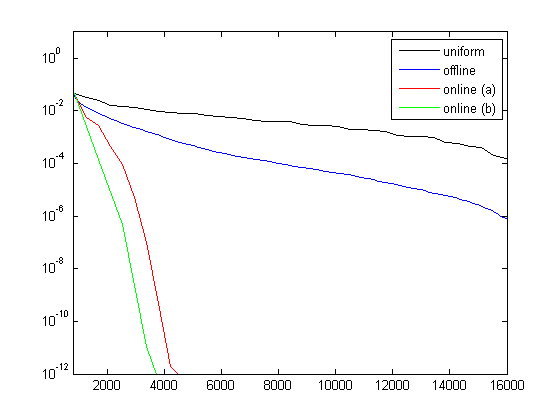}
        \caption{contrast = 1e-4} \label{fig:snapshot_error_plot_example_1_1e-4}
    \end{subfigure}
    \caption{Snapshot error of offline and online adaptive methods compared to the snapshot error of uniform enrichment for contrast 1e4 and 1e-4.}
    \label{fig:snapshot_error_plot_example_1}
\end{figure}

\begin{table}[!h]
    \centering
\hfill
    \begin{subtable}{.4\textwidth}
        \begin{tabular}{|c|c|c|}
             \hline
             DOF & $e$ (uniform) & $e$ (offline) \\ \hline
             2(840) & 0.0567 & 0.0567 \\ \hline
             8(3360) & 0.0123 & 0.0012 \\ \hline
             14(5880) & 0.0066 & 2.71e-4 \\ \hline
             20(8400) & 0.0040 & 7.57e-5 \\ \hline
             26(10920) & 0.0023 & 2.11e-5 \\ \hline
             32(13440) & 0.0011 & 6.01e-6 \\ \hline
	  38(15960) & 5.76e-4 & 4.61e-7 \\
             \hline
        \end{tabular}
        \caption{contrast = 1e4}\label{table:snapshot_error_table_offline_1e4}
    \end{subtable}
    \hfill
    \begin{subtable}{.4\textwidth}
        \begin{tabular}{|c|c|c|}
             \hline
             DOF & $e$ (uniform) & $e$ (offline) \\ \hline
             2(840) & 0.0452 & 0.0452 \\ \hline
             8(3360) & 0.0115 & 0.0017 \\ \hline
             14(5880) & 0.0059 & 2.75e-4 \\ \hline
             20(8400) & 0.0039 & 8.49e-5 \\ \hline
             26(10920) & 0.0019 & 2.98e-5 \\ \hline
             32(13440) & 9.57e-4 & 7.54e-6 \\ \hline
	  38(15960) & 1.60e-4 & 7.84e-7 \\
             \hline
        \end{tabular}
        \caption{contrast = 1e-4}\label{table:snapshot_error_table_offline_1e-4}
    \end{subtable}
\hfill
\null
    \caption{Snapshot error of the offline adaptive method compared to the GMsFEM with the same number of basis functions.}\label{table:snapshot_error_table_offline}
\end{table}

\begin{table}[!h]
    \centering
    \hfill
    \begin{subtable}{.35\textwidth}
        \begin{tabular}{|c|c|c|}
             \hline
             DOF & $e$ (a) & $e$ (b) \\ \hline
             2(840) & 0.0567 & 0.0567 \\ \hline
             3(1260) & 0.0065 & 0.0042 \\ \hline
             4(1680) & 0.0033 & 8.71e-5 \\ \hline
             5(2100) & 6.20e-4 & 4.08e-6 \\ \hline
             6(2520) & 2.56e-4 & 1.74e-7 \\ \hline
             7(2940) & 9.21e-6 & 2.81e-9 \\ \hline
	  8(3360) & 1.90e-7 & 3.42e-11 \\ \hline
	  9(3780) & 4.73e-10 & 5.58e-12 \\
             \hline
        \end{tabular}
        \caption{contrast = 1e4}\label{table:snapshot_error_table_online_1e4}
    \end{subtable}
    \hfill
    \begin{subtable}{.35\textwidth}
        \begin{tabular}{|c|c|c|}
             \hline
             DOF & $e$ (a) & $e$ (b) \\ \hline
             2(840) & 0.0452 & 0.0452 \\ \hline
             3(1260) & 0.0056 & 0.0045 \\ \hline
             4(1680) & 0.0028 & 1.36e-4 \\ \hline
             5(2100) & 4.56e-4 & 5.19e-6 \\ \hline
             6(2520) & 9.57e-5 & 4.94e-7 \\ \hline
             7(2940) & 4.97e-6 & 3.22e-9 \\ \hline
	  8(3360) & 1.03e-7 & 1.14e-11 \\ \hline
	  9(3780) & 4.84e-10 & 6.80e-13 \\
             \hline
        \end{tabular}
        \caption{contrast = 1e-4}\label{table:snapshot_error_table_online_1e-4}
    \end{subtable}
    \hfill
\null
    \caption{Snapshot error of the online adaptive method using two different choices of $\Omega_1, \dots ,\, \Omega_J$.}\label{table:snapshot_error_table_online}
\end{table}

From Figure \ref{fig:snapshot_error_plot_example_1}, we can see that the snapshot error from the adaptive methods is always smaller than the snapshot error obtained from uniformly increasing the number of basis functions. Moreover, the rate of convergence of the online adaptive method is faster than that of the offline adaptive method. Note that the online approaches require computations during the online stage of the simulations. From Table \ref{table:snapshot_error_table_offline}, we can also see the difference in the convergence rate between the offline adaptive method and uniform enrichment. This shows that the error indicator in the offline method can successfully show the coarse grid neighborhoods with insufficient bases. 

For the online adaptive method with the two choices of regions, the one with $2 \times 2$ coarse grid blocks give a faster convergence rate, as observed from the Table \ref{table:snapshot_error_table_online}. This can be explained by the number of coarse grid neighborhoods a region contains. In (a), each region contains one coarse grid neighborhood while in (b), each region contains four. Therefore, the space $\widehat{V}_{\Omega_i}$ for the calculation of the projection $\phi_j$ is larger in (b) than in (a) and captures more distant effects. Hence, the result is better in (b). 

\subsection{Online adaptive method with different number of initial basis functions}
In this example, we focus on the online adaptive method and we want to see the effect of using different number of initial basis functions in the method. We consider permeability fields $\kappa_2$ with the contrast 1e-4. For We divide the domain $[0,\, 1]^2$ into $8 \times 8$ coarse grids and divide each coarse grid into $32 \times 32$ fine grids. The source function $f$ is the same as the previous example. In each enrichment level, the regions $\Omega_1, \dots ,\, \Omega_J$ are chosen to be disjoint coarse grid neighborhoods.

We solve the equation using 1, 2, 3 and 4 initial basis functions obtained from solving the two spectral problems. We plot the snapshot error $e$ against the number of basis functions used in Figure \ref{fig:eg2_error_plot} and the value of $e$ is shown in Tables \ref{table:eg2_error_table_sp1} and \ref{table:eg2_error_table_sp2}. 

\begin{figure}[!h]
    \centering
    \begin{subfigure}[t]{0.49\textwidth}
        \includegraphics[width=\textwidth]{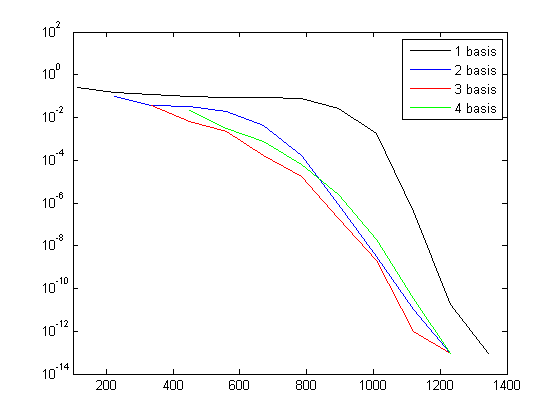}
        \caption{spectral problem 1}
        \label{fig:eg2_error_plot_sp1}
    \end{subfigure}
    \begin{subfigure}[t]{0.49\textwidth}
        \includegraphics[width=\textwidth]{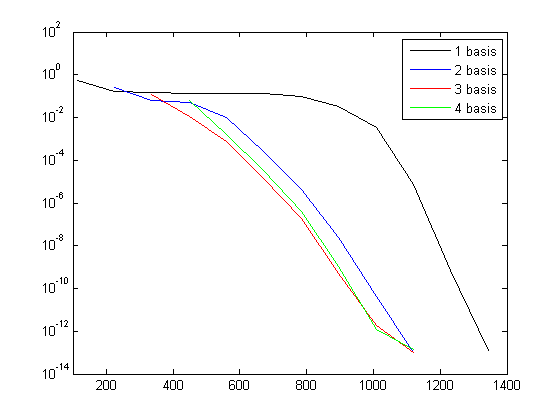}
        \caption{spectral problem 2}
        \label{fig:eg2_error_plot_sp2}
    \end{subfigure}
    \caption{Snapshot error of online adaptive method with different number of initial basis functions.}
    \label{fig:eg2_error_plot}
\end{figure}

\begin{table}[!h]
    \centering
        \begin{tabular}{|c|c|c|c|c|}
            \hline
            DOF & $e$ (1 basis) & $e$ (2 basis) & $e$ (3 basis) & $e$ (4 basis) \\ \hline
            1(112) & 0.2575 & / & / & / \\ \hline
            2(224) & 0.1454 & 0.0928 & / & / \\ \hline
            3(336) & 0.1254 & 0.0371 & 0.0357 & / \\ \hline
            4(448) & 0.0980 & 0.0326 & 0.0062 & 0.0211 \\ \hline
            5(560) & 0.0887 & 0.0200 & 0.0023 & 0.0031 \\ \hline
            6(672) & 0.0880 & 0.0042 & 1.68e-4 & 7.29e-4 \\ \hline
            7(784) & 0.0737 & 1.73e-4 & 1.68e-5 & 6.02e-5 \\ \hline
            8(896) & 0.0269 & 8.65e-7 & 1.87e-7 & 2.56e-6 \\ \hline
            9(1008) & 0.0019 & 3.03e-9 & 1.99e-9 & 2.04e-8 \\ \hline
        \end{tabular}
    \caption{Snapshot error of the online adaptive method using spectral problem 1 with 1 initial bases ($\Lambda_{\text{min}} = 0.0093$), 2 initial bases ($\Lambda_{\text{min}} = 0.0146$), 3 initial bases ($\Lambda_{\text{min}} = 2.5183$) and 4 initial bases ($\Lambda_{\text{min}} = 5.0668$). Contrast is 1e-4.}
\label{table:eg2_error_table_sp1}
\end{table}

\begin{table}[!h]
    \centering
        \begin{tabular}{|c|c|c|c|c|}
            \hline
            DOF & $e$ (1 basis) & $e$ (2 basis) & $e$ (3 basis) & $e$ (4 basis) \\ \hline
            1(112) & 0.5237 & / & / & / \\ \hline
            2(224) & 0.1684 & 0.2616 & / & / \\ \hline
            3(336) & 0.1417 & 0.0612 & 0.1215 & / \\ \hline
            4(448) & 0.1319 & 0.0498 & 0.0116 & 0.0625 \\ \hline
            5(560) & 0.1293 & 0.0098 & 7.71e-4 & 0.0017 \\ \hline
            6(672) & 0.1291 & 2.48e-4 & 1.46e-5 & 3.49e-5 \\ \hline
            7(784) & 0.0954 & 4.39e-6 & 1.84e-7 & 4.10e-7 \\ \hline
            8(896) & 0.0331 & 2.51e-8 & 5.31e-10 & 1.08e-9 \\ \hline
	 9(1008) & 0.0034 & 4.25e-11 & 1.83e-12 & 1.17e-12 \\ \hline
        \end{tabular}
    \caption{Snapshot error of the online adaptive method using spectral problem 2 with 1 initial bases ($\Lambda_{\text{min}} = 0.0016$), 2 initial bases ($\Lambda_{\text{min}} = 0.0247$), 3 initial bases ($\Lambda_{\text{min}} = 0.4939$) and 4 initial bases ($\Lambda_{\text{min}} = 0.7881$).  Contrast is 1e-4.}
\label{table:eg2_error_table_sp2}
\end{table}

From Figure \ref{fig:eg2_error_plot}, one can observe that if we use only 1 initial basis function, the rate of convergence is slow at the beginning. Similar behaviour can be seen if 2 initial basis functions are used, yet the convergence is faster. Using 3 initial basis functions seems to be the optimal choice in the sense that the snapshot error cannot be smaller when more initial basis functions are used. This can be explained by the value $\Lambda_{\text{min}}$. In the online adaptive method, this value depends on the initial basis functions obtained from the spectral problems. Theorem \ref{theorem_online} shows that the rate of convergence is bounded above by a value which decreases when $\Lambda_{\text{min}}$ increases. When spectral problem 1 is used, the values of $\Lambda_{\text{min}}$ are 0.0093, 0.0146, 2.5183 and 5.0068 when 1, 2, 3 and 4 initial basis functions are used respectively. When spectral problem 2 is used, the values of $\Lambda_{\text{min}}$ are 0.0016, 0.0247, 0.4939 and 0.7881 when 1, 2, 3 and 4 initial basis functions are used respectively. Therefore, the values of $\Lambda_{\text{min}}$ corresponding to the first two basis functions are small. This suggests a criterion for choosing the initial number of basis function which is to include all basis functions with small eigenvalue from the spectral problem.

Remark that the magnitude of $\Lambda_{\text{min}}$ depends on the choice of the spectral problem. For spectral problem 1, both the constant $C_{\text{err}}$ and $\Lambda_{\text{min}}$ grow with the ratio $H/h$. For spectral problem 2, $C_{\text{err}}$ is independent of the mesh size and $\Lambda_{\text{min}}$ is always bounded above by 1.

\subsection{Online adaptive method with different contrasts}
Next, we want to see the effect of varying the contrast to the online adaptive method. Similar to the previous example, we will start with different number of initial basis functions. Equation (\ref{equation}) is solved in permeability field $\kappa_1$ with three different contrasts 1e-2, 1e-4 and 1e-6. The coarse grid size is $15 \times 15$ and the fine grid size is $40 \times 40$ on each coarse grid. The source function $f$ is, again, $1$ at top left corner and $-1$ at bottom right corner. For each number of initial basis functions, we plot the snapshot error of the method with the three different contrasts (Figures \ref{fig:eg3_error_plot_sp1} and \ref{fig:eg3_error_plot_sp2}). We also list the value of the snapshot error in Tables \ref{table:eg3_error_table_sp1} and \ref{table:eg3_error_table_sp2}.

\begin{figure}[!h]
    \centering
    \begin{subfigure}{0.47\textwidth}
        \includegraphics[width=\textwidth]{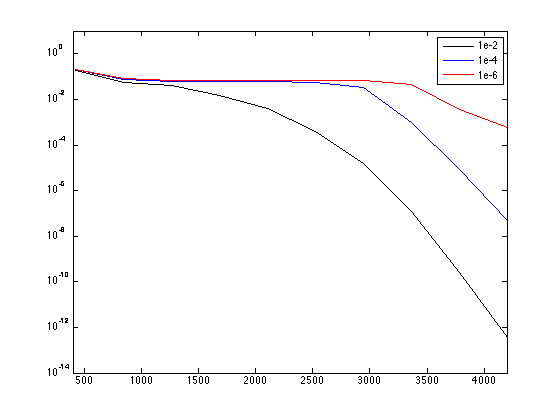}
        \caption{1 initial basis function.  }
        \label{fig:eg3_error_plot_sp1_1}
    \end{subfigure}
    \hfill
    \begin{subfigure}{0.47\textwidth}
        \includegraphics[width=\textwidth]{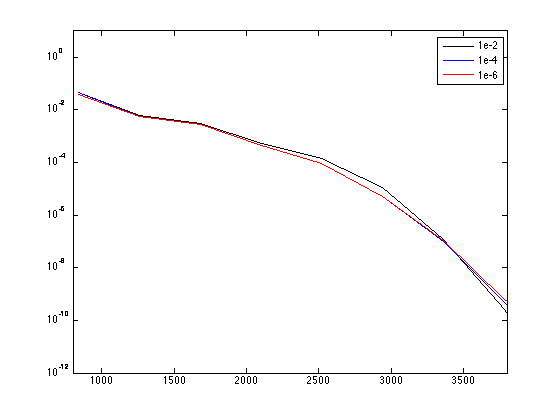}
        \caption{2 initial basis functions. }
        \label{fig:eg3_error_plot_sp1_2}
    \end{subfigure}
    
    \begin{subfigure}{0.47\textwidth}
        \includegraphics[width=\textwidth]{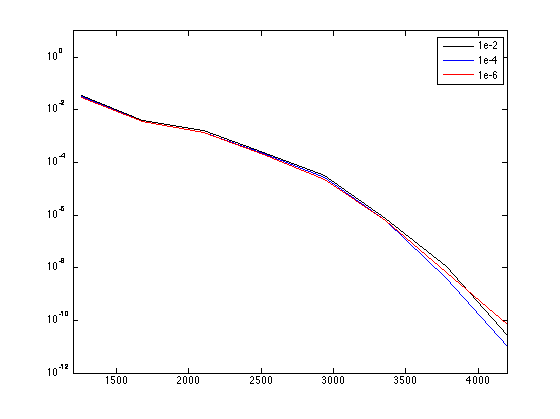}
        \caption{3 initial basis functions. }
        \label{fig:eg3_error_plot_sp1_3}
    \end{subfigure}
    \hfill
    \begin{subfigure}{0.47\textwidth}
        \includegraphics[width=\textwidth]{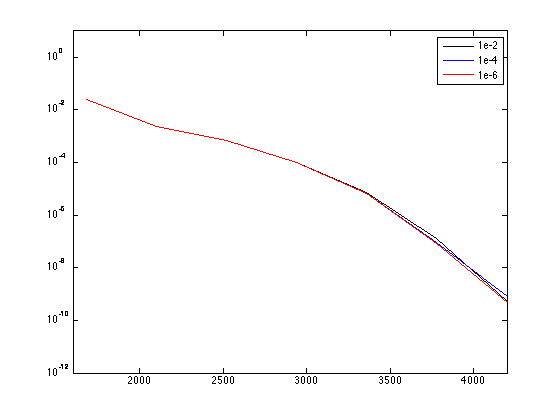}
        \caption{4 initial basis functions. }
        \label{fig:eg3_error_plot_sp1_4}
    \end{subfigure}
    \caption{Snapshot error of online adaptive method using spectral problem 1 with different number of initial basis functions and different contrast. }
    \label{fig:eg3_error_plot_sp1}
\end{figure}

\begin{table}[h!]
    \centering
    \hfill
    \begin{subtable}{.4\textwidth}
    \begin{tabular}{|c|c|c|c|}
        \hline
        DOF & 1e-2 & 1e-4 & 1e-6 \\ \hline
        1(420) & 0.1772 & 0.2039 & 0.2130 \\ \hline
        2(840) & 0.0560 & 0.0753 & 0.0839 \\ \hline
        3(1260) & 0.0399 & 0.0620 & 0.0699 \\ \hline
        4(1680) & 0.0151 & 0.0597 & 0.0676 \\ \hline
        5(2100) & 0.0039 & 0.0595 & 0.0674 \\ \hline
        6(2520) & 3.73e-4 & 0.0567 & 0.0674 \\ \hline
        7(2940) & 1.47e-4 & 0.0317 & 0.0643 \\ \hline
        8(3360) & 1.22e-7 & 0.0010 & 0.0466 \\ \hline
        9(3780) & 2.43e-10 & 8.41e-6 & 0.0035 \\ \hline
    \end{tabular}
    \caption{1 initial basis function. Values of $\Lambda_{\text{min}}$ are 0.0776, 0.0027 and 5.12e-5 for contrasts 1e-2, 1e-4 and 1e-6, respectively.}
    \label{table:eg3_error_table_sp1_1}
    \end{subtable}
    \hfill
    \begin{subtable}{.4\textwidth}
    \begin{tabular}{|c|c|c|c|}
        \hline
        DOF & 1e-2 & 1e-4 & 1e-6 \\ \hline
        2(840) & 0.0472 & 0.0452 & 0.0399 \\ \hline
        3(1260) & 0.0060 & 0.0056 & 0.0054 \\ \hline
        4(1680) & 0.0030 & 0.0028 & 0.0027 \\ \hline
        5(2100) & 5.47e-4 & 4.56e-4 & 4.42e-4 \\ \hline
        6(2520) & 1.41e-4 & 9.57e-5 & 9.04e-5 \\ \hline
        7(2940) & 1.11e-5 & 4.97e-6 & 5.29e-6 \\ \hline
        8(3360) & 1.20e-7 & 1.03e-7 & 1.10e-7 \\ \hline
        9(3780) & 2.67e-10 & 4.84e-10 & 6.16e-10 \\ \hline
    \end{tabular}
    \caption{2 initial basis functions. Values of $\Lambda_{\text{min}}$ are 1.9511, 1.9264 and 1.9262 for  contrasts 1e-2, 1e-4 and 1e-6, respectively. }
    \label{table:eg3_error_table_sp1_2}
    \end{subtable}
    \hfill
    \null
    
    \hfill
    \begin{subtable}{.4\textwidth}
    \begin{tabular}{|c|c|c|c|}
        \hline
        DOF & 1e-2 & 1e-4 & 1e-6 \\ \hline
        3(1260) & 0.0341 & 0.0317 & 0.0303 \\ \hline
        4(1680) & 0.0041 & 0.0037 & 0.0037 \\ \hline
        5(2100) & 0.0016 & 0.0014 & 0.0014 \\ \hline
        6(2520) & 2.40e-4 & 2.15e-4 & 2.02e-4 \\ \hline
        7(2940) & 3.26e-5 & 2.74e-5 & 2.19e-5 \\ \hline
        8(3360) & 7.28e-7 & 6.45e-7 & 6.13e-7 \\ \hline
        9(3780) & 1.15e-8 & 3.86e-9 & 6.83e-9 \\ \hline
    \end{tabular}
    \caption{3 initial basis functions. Values of $\Lambda_{\text{min}}$ are 3.6204, 3.5980 and 3.5978 for  contrasts 1e-2, 1e-4 and 1e-6, respectively.}
    \label{table:eg3_error_table_sp1_3}
    \end{subtable}
    \hfill
    \begin{subtable}{.4\textwidth}
    \begin{tabular}{|c|c|c|c|}
        \hline
        DOF & 1e-2 & 1e-4 & 1e-6 \\ \hline
        4(1680) & 0.0257 & 0.0240 & 0.0239 \\ \hline
        5(2100) & 0.0024 & 0.0024 & 0.0024 \\ \hline
        6(2520) & 6.92e-4 & 6.98e-4 & 6.99e-4 \\ \hline
        7(2940) & 9.78e-5 & 1.01e-4 & 9.92e-5 \\ \hline
        8(3360) & 7.32e-6 & 6.63e-6 & 6.62e-6 \\ \hline
        9(3780) & 1.26e-7 & 9.01e-8 & 8.12e-8 \\ \hline
    \end{tabular}
    \caption{4 initial basis functions. Values of $\Lambda_{\text{min}}$ are 5.2765, 5.2396 and 5.2656 for  contrasts 1e-2, 1e-4 and 1e-6, respectively.}
    \label{table:eg3_error_table_sp1_4}
    \end{subtable}
    \hfill
\null
    \caption{Snapshot error of online adaptive method using spectral problem 1 with different number of initial basis functions and different contrast.}
    \label{table:eg3_error_table_sp1}
\end{table}

From Figure \ref{fig:eg3_error_plot_sp1}, we can see that when spectral problem 1 is used, the change in the contrast has almost no effect on the snapshot error if we start with 2, 3 or 4 basis functions on each coarse grid neighborhood. This can also be confirmed by looking at Table \ref{table:eg3_error_table_sp1}. However, if we use only 1 initial basis function, the contrast makes a huge difference. The convergence rate decreases as the contrast changes from 1e-2 to 1e-4 and then 1e-6. This result can also be explained by the value of $\Lambda_{\text{min}}$. The captions in Table \ref{table:eg3_error_table_sp1} list the values of $\Lambda_{\text{min}}$ corresponding to different number of initial basis functions and different contrast value. These values are consistent with Figure \ref{fig:eg3_error_plot_sp1} since we can observe almost no changes in $\Lambda_{\text{min}}$ as the contrast varies if we use 2, 3 or 4 initial basis functions, while $\Lambda_{\text{min}}$ decreases with the contrast for the case of 1 initial basis function.

\begin{figure}[!h]
    \centering
    \begin{subfigure}{0.47\textwidth}
        \includegraphics[width=\textwidth]{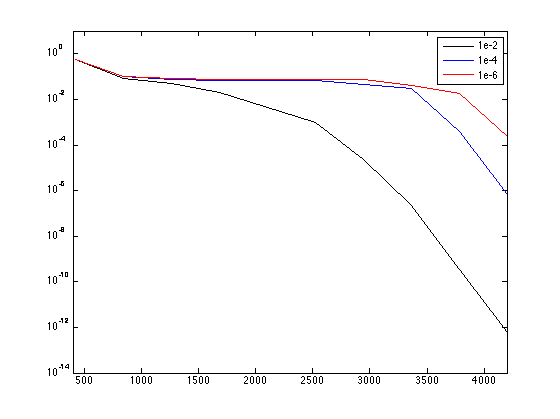}
        \caption{1 initial basis function. }
        \label{fig:eg3_error_plot_sp2_1}
    \end{subfigure}
    \hfill
    \begin{subfigure}{0.47\textwidth}
        \includegraphics[width=\textwidth]{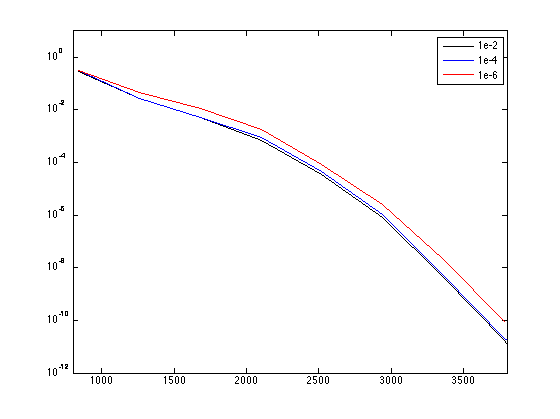}
        \caption{2 initial basis functions.  }
        \label{fig:eg3_error_plot_sp2_2}
    \end{subfigure}
    
    \begin{subfigure}{0.47\textwidth}
        \includegraphics[width=\textwidth]{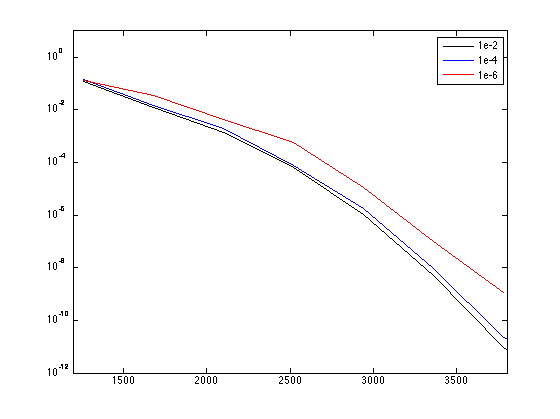}
        \caption{3 initial basis functions.  }
        \label{fig:eg3_error_plot_sp2_3}
    \end{subfigure}
    \hfill
    \begin{subfigure}{0.47\textwidth}
        \includegraphics[width=\textwidth]{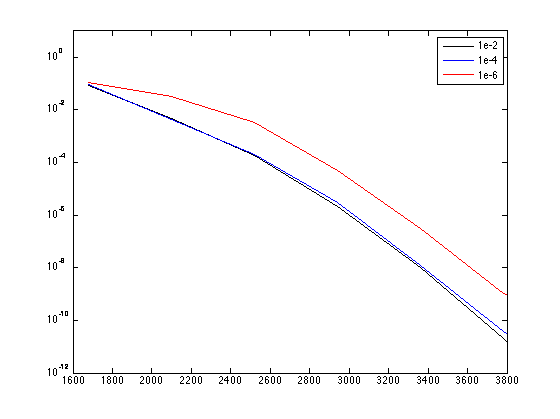}
        \caption{4 initial basis functions.  }
        \label{fig:eg3_error_plot_sp2_4}
    \end{subfigure}
    \caption{Snapshot error of online adaptive method using spectral problem 2 with different number of initial basis functions and different contrast. }
    \label{fig:eg3_error_plot_sp2}
\end{figure}

\begin{table}[h!]
    \centering
    \hfill
    \begin{subtable}{.4\textwidth}
    \begin{tabular}{|c|c|c|c|}
        \hline
        DOF & 1e-2 & 1e-4 & 1e-6 \\ \hline
        1(420) & 0.5506 & 0.5716 & 0.5810 \\ \hline
        2(840) & 0.0799 & 0.0968 & 0.1054 \\ \hline
        3(1260) & 0.0477 & 0.0719 & 0.0803 \\ \hline
        4(1680) & 0.0021 & 0.0692 & 0.0778 \\ \hline
        5(2100) & 0.0045 & 0.0690 & 0.0776 \\ \hline
        6(2520) & 9.73e-4 & 0.0649 & 0.0776 \\ \hline
        7(2940) & 2.23e-5 & 0.0463 & 0.0735 \\ \hline
        8(3360) & 2.20e-7 & 0.0298 & 0.0401 \\ \hline
        9(3780) & 3.42e-10 & 4.1024e-4 & 0.0184 \\ \hline
    \end{tabular}
    \caption{1 initial basis function. Values of $\Lambda_{\text{min}}$ are 0.0510, 9.41e-4 and 4.33e-5 for 1e-2, 1e-4 and 1e-6 respectively. }
    \label{table:eg3_error_table_sp2_1}
    \end{subtable}
    \hfill
    \begin{subtable}{.4\textwidth}
    \begin{tabular}{|c|c|c|c|}
        \hline
        DOF & 1e-2 & 1e-4 & 1e-6 \\ \hline
        2(840) & 0.2909 & 0.3057 & 0.3164 \\ \hline
        3(1260) & 0.0270 & 0.0283 & 0.0477 \\ \hline
        4(1680) & 0.0050 & 0.0052 & 0.0116 \\ \hline
        5(2100) & 6.66e-4 & 9.15e-4 & 0.0017 \\ \hline
        6(2520) & 3.44e-5 & 4.67e-5 & 8.54e-5 \\ \hline
        7(2940) & 8.12e-7 & 1.02e-6 & 2.53e-6 \\ \hline
        8(3360) & 4.13e-9 & 5.20e-9 & 2.06e-8 \\ \hline
        9(3780) & 1.76e-11 & 2.11e-11 & 8.78e-11 \\ \hline
    \end{tabular}
    \caption{2 initial basis functions. Values of $\Lambda_{\text{min}}$ are 0.3176, 0.3823 and 0.3790 for 1e-2, 1e-4 and 1e-6 respectively.}
    \label{table:eg3_error_table_sp2_2}
    \end{subtable}
    \hfill
    \null
    
    \hfill
    \begin{subtable}{.4\textwidth}
    \begin{tabular}{|c|c|c|c|}
        \hline
        DOF & 1e-2 & 1e-4 & 1e-6 \\ \hline
        3(1260) & 0.1202 & 0.1418 & 0.1319 \\ \hline
        4(1680) & 0.0128 & 0.0148 & 0.0357 \\ \hline
        5(2100) & 0.0014 & 0.0020 & 0.0043 \\ \hline
        6(2520) & 6.66e-5 & 7.69e-5 & 5.68e-4 \\ \hline
        7(2940) & 1.02e-6 & 1.73e-6 & 1.14e-5 \\ \hline
        8(3360) & 5.19e-9 & 9.20e-9 & 9.76e-8 \\ \hline
        9(3780) & 9.61e-12 & 2.29e-11 & 1.13e-9 \\ \hline
    \end{tabular}
    \caption{3 initial basis functions. Values of $\Lambda_{\text{min}}$ are 0.6460, 0.4882 and 0.5763 for 1e-2, 1e-4 and 1e-6 respectively.}
    \label{table:eg3_error_table_sp2_3}
    \end{subtable}
    \hfill
    \begin{subtable}{.4\textwidth}
    \begin{tabular}{|c|c|c|c|}
        \hline
        DOF & 1e-2 & 1e-4 & 1e-6 \\ \hline
        4(1680) & 0.0824 & 0.0930 & 0.1128 \\ \hline
        5(2100) & 0.0046 & 0.0044 & 0.0315 \\ \hline
        6(2520) & 1.84e-4 & 1.99e-4 & 0.0035 \\ \hline
        7(2940) & 2.17e-6 & 3.15e-6 & 4.84e-5 \\ \hline
        8(3360) & 1.05e-8 & 1.23e-8 & 3.25e-7 \\ \hline
        9(3780) & 2.02e-11 & 3.65e-11 & 1.02e-9 \\ \hline
    \end{tabular}
    \caption{4 initial basis functions. Values of $\Lambda_{\text{min}}$ are 0.7697, 0.7650 and 0.7307 for 1e-2, 1e-4 and 1e-6 respectively.}
    \label{table:eg3_error_table_sp2_4}
    \end{subtable}
    \hfill
\null
    \caption{Snapshot error of online adaptive method using spectral problem 2 with different number of initial basis functions and different contrast.}
    \label{table:eg3_error_table_sp2}
\end{table}

Similar behaviour can be observed when spectral problem 2 is used. We can observe from Table \ref{fig:eg3_error_plot_sp2} that if we start with one basis function per coarse grid neighborhood, the larger the contrast, the slower is the convergence rate. However, Figure \ref{fig:eg3_error_plot_sp2} also suggests that spectral problem 2 is less resistent to changes in the contrast than spectral problem 1. When 2, 3 or 4 initial basis functions are used, the snapshot error is nearly the same for contrast 1e-2 and 1e-4, yet the error increases as the contrast changes from 1e-4 to 1e-6. This jump is larger for 4 initial basis functions than 2 initial basis functions. Although spectral problem 2 is not as good as spectral problem 1 in the aspect of the resistance to changes in contrast, by comparing Tables \ref{table:eg3_error_table_sp1} and \ref{table:eg3_error_table_sp2}, we can see that the convergence rate is faster using spectral problem 2 than spectral problem 1.

From the above observations, we know that the convergence rate of the online adaptive method can be independent of the contrast if we choose the initial basis functions well. Basis functions corresponding to small value of $\Lambda_{\text{min}}$ are contrast-dependent, and we should include them in the initial basis. When $\Lambda_{\text{min}}$ is large, we may expect that the online adaptive method will perform well with large contrast.

\section{Conclusion}

In this paper, we present two adaptive enrichment algorithms for the mixed GMsFEM. The first one is an offline adaptive method which adds basis computed from the offline stage. This method is based on a local error indicator which is the norm of a residual operator restricted on a local space. Offline basis functions are added to those coarse grid neighborhoods with large errors. The other algorithm is an online adaptive method. Online basis functions are constructed by solving local problems based on a residual operator. We select non-overlapping regions in the domain and solve for an online basis function on each of the regions.

We show theoretically the convergence of the two methods. For the offline method, the rate of convergence depends on two parameters, which control the number of coarse grid neighborhoods to be selected to add basis functions and the number of basis functions to be added. The convergence rate of the online method depends on the eigenvalues corresponding to those basis functions that are not included in the initial basis. The larger the above mentioned eigenvalues, the faster is the convergence rate. The numerical results are consistent with these findings. It is also shown numerically that if the basis functions corresponding to the eigenvalues that are contrast-dependent are included in the initial basis, the online adaptive method will be resistant to the change in the contrast value. Those eigenvalues are the smallest ones and therefore the corresponding basis functions should be included in the initial basis to speed up the convergence.

\nocite{*}

\bibliographystyle{plain}
\bibliography{reference_inline}

\end{document}